\documentclass[a4paper]{amsart}

\usepackage{amssymb}
\usepackage{amsthm}  
\usepackage{amsmath} 
\usepackage{amscd} 
\usepackage[all]{xypic}
\usepackage{url}
\usepackage{multirow}
\usepackage{hyperref}
\hypersetup{linktocpage}
\usepackage{color}
\usepackage{nicematrix}

\newcommand{\bg}{\boldsymbol{g}}

\newcommand{\om}{\omega}

\newcommand{\ba}{\mathcal{G}}

\newcommand{\fg}{\frak g}

\newcommand{\fp}{\frak p}

\newcommand{\fk}{\frak k}

\newcommand{\fl}{\frak l}
\newcommand{\fh}{\frak h}
\newcommand{\fn}{\frak n}

\newcommand{\s}{\mathrm{\varsigma}}
\newcommand{\fa}{\frak a}

\newcommand{\fc}{\frak c}
\newcommand{\so}{\frak{so}(p+1,q+1)}
\newcommand{\sodc}{\frak{so}(2,4)}

\newcommand{\Ad}{{\rm Ad}}

\newcommand{\id}{{\rm id}}

\newcommand{\na}{\nabla}

\newcommand{\Rho}{{\mbox{\sf P}}}
\newcommand{\R}{\mathbb{R}}
\newcommand{\C}{\mathbb{C}}

\newtheorem{thm}{Theorem}[section]
\newtheorem{prop}{Proposition}[section]
\newtheorem*{prop*}{Proposition}
\newtheorem*{thm*}{Theorem}
\newtheorem{lem}{Lemma}[section]
\newtheorem*{lemma*}{Lemma}
\newtheorem{cor}{Corollary}[section]
\newtheorem*{cor*}{Corollary}

\newtheorem*{rem*}{Remark}
\theoremstyle{definition}
\newtheorem{def*}{Definition}[section]

\theoremstyle{remark}
\newtheorem{exam}{Example}

\begin{document}
\title{First BGG operators on homogeneous conformal geometries}
\author{Jan Gregorovi\v c and Lenka Zalabov\' a}
\address{J.G.: Institute of Discrete Mathematics and Geometry, TU Vienna, Wiedner Hauptstrasse 8-10/104, 1040 Vienna, Austria, and
    Department of Mathematics and Statistics, 
	Faculty of Science, Masaryk University,
	Kotl\' a\v rsk\' a 2, 611 37 Brno, Czech Republic  L.Z.: Institute of Mathematics, 
	Faculty of Science, University of South Bohemia, 
	Brani\v sovsk\' a 1760, 370 05 \v Cesk\' e Bud\v ejovice, and
    Department of Mathematics and Statistics, 
	Faculty of Science, Masaryk University,
	Kotl\' a\v rsk\' a 2, 611 37 Brno, Czech Republic}
 \email{jan.gregorovic@seznam.cz, lzalabova@gmail.com}
 \thanks{J.G. is supported by the Austrian Science Fund (FWF): P34369. L.Z. is supported by the grant GACR 20-11473S: Symmetry and invariance in analysis, geometric modelling and control theory. }

\begin{abstract}
We study first BGG operators and their solutions on homogeneous conformal geometries. We focus on conformal Killing tensors, conformal Killing--Yano forms and twistor spinors in particular.
We develop an invariant calculus that allows us to find solutions explicitly using only algebraic computations.
We also discuss applications to holonomy reductions and conserved quantities of conformal circles.
We demonstrate our result on examples of homogeneous conformal geometries coming mostly from general relativity.
\end{abstract}
\maketitle

\section{Introduction}
In this article, we consider homogeneous conformal structure  $[g]$ of signature $(p,q)$ on a connected smooth manifold $M$ of dimension $n=p+q$. This is an equivalence class $[g]$ of pseudo--Riemannian metrics of signature $(p,q)$ spanning a line subbundle $\mathcal{E}[-2]\subset \bigodot^2 T^*M$ such that the group of conformal symmetries of $[g]$ acts transitively on $M$. The first BGG operators form an important class of conformally invariant differential operators defined on specific bundles $\mathcal{X}\to M$ on conformal manifolds, \cite{CD,BCEG,C-overdetermined,H-srni}. For example, conformal Killing tensors, conformal Killing--Yano forms, or twistor spinors, i.e., conformal Killing spinors, are solutions of first BGG operators on particular bundles $\mathcal{X}$, \cite{Y,PR,Kress,Sem,black-holes}. There is also a distinguished class of normal solutions of the first BGG operators, \cite{CGH1}. 

We develop an invariant calculus for first BGG operators on homogeneous conformal structures $(M,[g])$. This is a distinguished case of general calculus for homogeneous parabolic geometries studied in \cite{bgg-my} that we decided to study separately for its importance.  One can expect that due to the homogeneity, all the computations can be pulled back to a single point and then encoded as algebraic objects. Then the main result of the invariant calculus we present here can be summarized as follows.

\begin{thm}\label{thm-intro}
Let $(M,[g])$ be a conformal geometry with signature $(p,q)$ and consider a first BGG operator on a  bundle $\mathcal{X}\to M$. Suppose $K$ is a Lie group of conformal symmetries of $[g]$ acting transitively on $M$ and let $H$ be the stabilizer of a point $x\in M$. Then there are representations $\phi: K\to Gl(\mathbb{G})$ and  $\mu: H\to Gl(\mathbb{X})$ and an $H$--equivariant projection $\pi: \mathbb{G}\to \mathbb{X}$ such that $$(\mathcal{X}\to M)\cong (K\times_{\mu(H)} \mathbb{X}\to K/H),$$ and for each $v\in \mathbb{G}$, the section of $\mathcal{X}$ induced by the function
\[
K\to \mathbb{X},\ \   k\mapsto \pi( \phi(k)^{-1}(v))
\]
is a solution of the first BGG operator  on a  bundle $\mathcal{X}\to M$.
\end{thm}

In fact, we prove in Theorem \ref{thm-loc} a local version of this theorem that provides representations $\Phi: \fk\to \frak{gl}(\mathbb{S})$ and $d\mu: \fh\to \frak{gl}(\mathbb{X})$ and an $\fh$--equivariant projection $\pi: \mathbb{S}\to \mathbb{X}$ describing the set $\mathbb{S}$ of  local solutions of the first BGG operators, where $\fk,\fh$ are the Lie algebras of $K, H$, respectively. The maximal subrepresentation $\mathbb{G}\subset \mathbb{S}$ of $\fk$ that can be integrated to a representation $\phi$ of $K$ then provides the global solutions.

We summarize in Section \ref{section3.3} an algorithm for computing of the solutions of the first BGG operators using Theorems \ref{thm-intro} and \ref{thm-loc}. Let us emphasize that the algorithm can be straightforwardly implemented using a computer algebra system and in fact, we use Maple for our computations. Let us present here several key facts.

The projection $\pi$, as we recall in  Section \ref{sec3}, comes from the tractorial approach to BGG operators, \cite{H-srni,GC,thomas}, where each first BGG operator is encoded by an (irreducible) representation $$\rho: \so\to \frak{gl}(\mathbb{V}).$$
Then $\mathbb{X}$ can be identified with the $\frak{co}(p,q)$--module of the lowest weight in $\mathbb{V}$ that is usually called the projective slot. Since the sets of solutions are identified with subsets $\mathbb{G}\subset \mathbb{S}\subset \mathbb{V}$, it follows that $\pi$ is induced by the natural projection $\mathbb{V}\to \mathbb{X}.$

The representations $\Phi$ and $d\mu$, as we show in the proof of Theorem \ref{thm-loc}, can be algebraically computed. The data we use in the computation are the representation $\rho$ and a linear map $$\alpha: \fk\to \so$$ called conformal extension that can be associated with each homogeneous conformal structure with signature $(p,q)$. In particular, the map $\alpha$ restricted to $\fh$ is a Lie algebra homomorphism $d\iota: \fh\to \fp$, where $\fp$ is the Poincar\' e Lie subalgebra $\fp$ of $\so$. Since the representation $\rho|_\fp$ can be restricted to $\mathbb{X}$, we obtain $d\mu=\rho|_\fp\circ d\iota$. 
Let us remark that the tractorial approach to BGG operators in the global setting involves topological obstructions requiring the existence of the Lie group $P$ with the Lie algebra $\fp$ for which both the Lie algebra homomorphisms $\rho|_\fp$ and $d\iota$ can be integrated. In our approach, the integrability of $d\mu$ follows from the existence of the bundle $\mathcal{X}$ and we just need to discuss the relation of the sets $\mathbb{G}$ and $\mathbb{S}$.

In Section \ref{section2} we introduce conformal extensions. We describe how to find the conformal extension associated with a homogeneous conformal structure in Proposition \ref{conf_ext}. We show in Sections \ref{loc_coord} and \ref{cartan} that conformal extensions provide a complete local description of both the conformal geometry and the associated Cartan geometry, \cite{parabook}. In particular, there is a distinguished complement $\fc$ of $\fh$ in $\fk$ that determines suitable exponential coordinates 
$${\sf c}: \fc \to M$$ 
and a coframe $e^*: T\fc\to \R^n$ that provides this local description explicitly. We illustrate this on the example of the conformal class given by the famous G\"odel metrics, \cite{Godel}.

We prove our main results in Section \ref{sec3}.  In Section \ref{section3.4} we discuss in detail how to use the exponential coordinates ${\sf c}$ and the coframe $e^*$ to describe the bundle isomorphism $(\mathcal{X}\to M)\cong (K\times_{\mu(H)} \mathbb{X}\to K/H)$ and the solutions of the first BGG operators in the local coordinates. We illustrate this in Proposition \ref{prop-sol-godel} by computing  the solutions of several first BGG operators for the conformal class of the G\"odel metrics. 

In Sections \ref{section4} and \ref{section5}, we compute solutions and normal solutions of the first BGG operators on further examples and also discuss several  of their applications.
In Section \ref{section4}, we discuss  holonomy reductions provided by normal solutions of first BGG operators, \cite{hol}. We provide three particular examples which are conformal classes of submaximally symmetric pp--wave of signature $(1,3)$, \cite{submax,kundt}, non--reductive version of G\"odel metric, and an invariant metric on the Lie group $Gl(2,\R).$ In the pp--wave case, we describe all normal solutions for all first BGG operators and determine the holonomy reductions induced by Einstein scales, twistor spinors, and normal conformal Killing fields. In the case of non--reductive G\"odel, we also describe all normal solutions for all  first BGG operators and determine the holonomy reductions induced by twistor spinors and normal conformal Killing fields. In the remaining case, we discuss the holonomy reduction provided by a normal conformal Killing field that relates this example to CR geometry, \cite{CG-Fefferman}.

In Section \ref{section5}, we discuss the application of conformal Killing--Yano $2$--forms for finding conserved quantities on conformal circles, \cite{BE,conserved,DT}. We  consider two examples which are conformal structure on the $3$--dimensional Heisenberg group, called $Nil$ in \cite{Tod}, and a symmetric space with the split version of the Fubini--Study metric. We compute conformal Killing--Yano $2$--forms on them, determine the conserved quantities, and discuss how to use them to obtain an explicit description of  the conformal circles.

Since the formulas for first BGG operators in local coordinates are not available in the literature in many cases, we describe in Appendix \ref{bgg-const} how to find such a  formula in the exponential coordinates ${\sf c}$. We also provide in  Appendix \ref{apendix} an explicit example of computation of the prolongation connection that is part of the algorithm from Section \ref{section3.3}.

\subsection*{Notation}

Let us fix some conventions that we use throughout the article. We fix the scalar product $\bg$ on $\mathbb{T}:=\R^{n+2}$, $n=p+q$, $p\leq q$, given by the $(1,p,q-p,p,1)$--block matrix 
$$
\bg=\left[ \begin{smallmatrix}
 0 & 0 & 0 & 0 & 1 \\
 0 & 0 & 0 & I_p & 0 \\
 0 & 0 & I_{q-p} & 0 & 0 \\
 0 & I_p & 0 & 0 & 0 \\
 1 & 0 & 0 & 0 & 0 
\end{smallmatrix} \right],$$
where $I_r$ denotes the identity matrix of order $r$.
We fix the Lie algebra $\so$ as the Lie subalgebra of $\frak{gl}(\mathbb{T})$ preserving $\bg$, i.e., $\so$ consists of elements
\begin{align}\label{blockso} \tag{SO}
\left[ \begin{smallmatrix}
 a & b & c & d & 0 \\
 u & A & E & F & -d^t \\
 v & C & B & -E^t & -c^t \\
 w & D & -C^t & -A^t & -b^t \\
 0 & -w^t & -v^t & -u^t & -a 
\end{smallmatrix} \right] ,
\end{align}
where $a \in \R$, $A \in \frak{gl}(p,\R)$, $B \in \frak{so}(q-p)$, $C, E^t \in \R^{p*} \otimes \R^{q-p}$, $ D,F \in \frak{so}(p)$, $u,w,b^t,d^t \in \R^p$ and $v,c^t \in \R^{q-p}$. 
We have several reasons for this convention and summarize them as follows.
\noindent \\[1mm] {$\bullet$}
The subalgebra $\R^n\subset \so$ consisting of $(u,v,w)$--parts of \eqref{blockso} together with the scalar product $$
\nu_{p,q}=\left[ \begin{smallmatrix}
  0 & 0 & I_p  \\
  0 & I_{q-p} & 0 \\
  I_p & 0 & 0  \\
\end{smallmatrix} \right]$$
provides a linear model of conformal geometries, i.e., conformal coframes are (local) maps $TM\to \R^n$ mapping the conformal class $[g]$ to functional multiples of $\nu_{p,q}$. 
\noindent \\[1mm] {$\bullet$}
 The subalgebra $\frak{co}(p,q)\subset \so$ consisting of $(a,A,B,C,D,E,F)$--parts of \eqref{blockso} is the Lie algebra of the conformal group $CO(p,q)$ preserving the scalar product $\nu_{p,q}$ up to  multiples.
\noindent \\[1mm] {$\bullet$}
 The remaining subalgebra $\R^{n*}\subset \so$ consisting of $(b,c,d)$--parts of \eqref{blockso} corresponds to the first prolongation of $\frak{co}(p,q)$. Note that the duality between $\R^n$ and $\R^{n*}$ is provided by the Killing form of $\so.$
\noindent \\[1mm] {$\bullet$}
The diagonal  of $\so$  forms a Cartan subalgebra and the corresponding root spaces are compatible with \eqref{blockso}. In particular, all negative real root spaces are below the diagonal and positive ones are above the diagonal and $\fp:=\frak{co}(p,q)\oplus \R^{n*}$ is the Poincar\'e (parabolic) algebra.
\noindent \\[1mm] 
 Since there will be formulas involving simultaneously Lie brackets from $\so$ and other Lie algebras, we use the notation $\{\:,\:\}$ for the Lie bracket on $\so$ and the notation $[\:,\:]$ for the Lie bracket on  other Lie algebras in question.
 
Further, we denote by $\bigodot^k_0 \mathbb{T}$ the highest weight component of the symmetrized product of $k$ copies of $\mathbb{T}$, and by $\wedge^k \mathbb{T}$ the skew--symmetrized product of $k$ copies of $\mathbb{T}$. We denote by $\mathbb{D}$ or $\mathbb{D}^{\pm}$ the spin or half--spin representations of $\so$, respectively. For a general tensor product $\mathbb{V}$ of such  representations, we denote by $\boxtimes(\mathbb{V})$ the highest weight (Cartan) component in this tensor product.

The restriction of $\rho: \so\to \frak{gl}(\mathbb{V})$ to $\fp$ induces a filtration $\mathbb{V}^i$ of $\mathbb{V}$, where $i$ are the half--integers determined by the action of the grading element which is given by $(1,n,1)$--block matrix 
\[
E:=\left[ \begin{smallmatrix}
1&0&0\\
0&0&0\\
0&0&-1
\end{smallmatrix} \right].
\]
Precisely, the eigenspaces $\mathbb{V}_j=\{X\in \mathbb{V}: \rho(E)(X)=jX\}$ define an associated grading $\mathbb{V}_j$ to the filtration $\mathbb{V}^j=\bigoplus_{k\geq j} \mathbb{V}_k$. Each $\mathbb{V}_j$ is a $\frak{co}(p,q)$--representation and, if $\rho$ is irreducible, then $\mathbb{X}$ is irreducible $\frak{co}(p,q)$--representation $\mathbb{V}_j$ with minimal $j$. Such  $\mathbb{X}$ is a tensor product of an irreducible representation of $\frak{so}(p,q)$ and one--dimensional $(-w)$--eigen representation $\R[w]$ of $E$.

Finally, the line bundle $\mathcal{E}[w]$ of conformal densities of weight $w$ is (if it exists) the $(-\frac{w}{n})$--th power of $\wedge^n T^*M$.

\section{Homogeneous conformal structures} \label{section2}

\subsection{Description of homogeneous conformal structures} 

Classically, a $K$--invariant conformal structure $[g]$ of signature $(p,q)$ on $M=K/H$ is encoded in a non--degenerate element 
\[g_o\in \bigodot^2 \fk/\fh^*\]
of signature $(p,q)$ that is preserved by the isotropy action of the stabilizer $H$ of $o\in M$ up to a positive multiple. Then a linear isomorphism $$\alpha_{-1}: T_oM=\fk/\fh\to \R^n$$ such that $g_o=\alpha_{-1}^*\nu_{p,q}$ induces a Lie group homomorphism $$\iota_0: H\to CO(p,q)$$
and identifies the bundle of conformal frames (i.e., the $CO(p,q)$-structure on $M$) with the bundle $K\times_{\iota_0(H)} CO(p,q)$. Thus the map $\alpha_{-1}$ contains all the information about the conformal geometry, but to encode all the information we need, we extend $\alpha_{-1}$ in the following way.

\begin{prop}\label{conf_ext}
Let $(M,[g])$ be a $K$--homogeneous conformal geometry  and  $H\subset K$ a stabilizer of a point $o\in M$. Let $\alpha_{-1}: T_oM=\fk/\fh\to \R^n$ be a linear isomorphisms such that $\alpha_{-1}^*\nu_{p,q}\in [g]_o$. Then there is linear map $\alpha: \fk\to \so$ such that
\begin{enumerate}
\item  $\{\alpha(Y_1),\alpha(Y_h)\}-\alpha([Y_1,Y_h])=0$ for all $Y_1\in \fk,Y_h\in \fh$,
\item $\alpha(\fk/\fh)=\so/\fp$ and moreover, the restriction of $\alpha$ to a map $\fk/\fh\to \R^n$ along $\fp$ coincides with $\alpha_{-1}$,
\item the curvature
\begin{align*}
\kappa(\alpha(Y_1)+\fp,\alpha(Y_2)+\fp):=\{\alpha(Y_1),\alpha(Y_2)\}-\alpha([Y_1,Y_2])
\end{align*}
in $\wedge^2 \R^{n*}\otimes \so$ satisfies the normalization condition
\begin{align}\label{norm}\tag{Nor}
 \sum_i \{Z_i,\kappa(\alpha(Y)+\fp,\alpha(X_i)+\fp)\}=0
\end{align}
 for all $Y\in \fk$, where the elements $X_i\in \fk$ are representatives of a basis $\alpha_{-1}(X_i)$ of $\R^n$ and the elements $Z_i\in \R^{n*}$ form the corresponding dual basis.
\end{enumerate}
\end{prop}
\begin{proof}
The existence of such $\alpha$ follows from \cite[Theorem 1.5.15]{parabook} and \cite[Theorem 1.6.7]{parabook}. We provide here an explicit construction of such $\alpha$ we will use later. Let us emphasize that all the conditions listed below lead to linear equations that can be always solved.

We pick a complement $\fc$ of $\fh$ in $\fk$ and get a map $\alpha_{-1}: \fc\to \R^n\subset \so$. Without loss of generality, we can assume $d\iota_0: \fh\to \frak{co}(p,q)\subset \so$ is injective because otherwise, there is a  conformal Killing field with higher--order fixed point which makes the conformal geometry flat. Then $\fk$ is a Lie subalgebra of $\so$ and $\alpha$ is just the inclusion. This altogether provides an associated graded map 
$$\alpha_{-1}+d\iota_0: \fk=\fc\oplus \fh\to \so.$$
Next, we extend $d\iota_0$ to a map $d\iota: \fh\to \fp$ by  a linear map $\fh\to \R^{n*}$ such that the component of $$\{\alpha_{-1}(Y_1),d\iota(Y_h)\}-(\alpha_{-1}+d\iota)([Y_1,Y_h])$$ in $\frak{co}(p,q)$ vanishes. This provides an injective Lie algebra homomorphism $d\iota: \fh\to \fp$ that satisfies  $\{\alpha_{-1}(Y_1),d\iota(Y_h)\}-(\alpha_{-1}+d\iota)([Y_1,Y_h])=0$ modulo $\R^{n*}$ for all $Y_1\in \fk,Y_h\in \fh$.

Further, there is a linear map $\alpha_{0}: \fc\to \frak{co}(p,q)$ such that $$\{(\alpha_{-1}+\alpha_{0})(Y_1),d\iota(Y_h)\}-(\alpha_{-1}+\alpha_0+d\iota)([Y_1,Y_h])=0$$ for all $Y_1\in \fk,Y_h\in \fh$ and such that the component of the normalization condition  \eqref{norm} in $\frak{co}(p,q)$ computed for 
for $\alpha_{-1}+\alpha_0+d\iota$ instead of $\alpha$ vanishes. Note that $\alpha_0$ is not unique and one can usually observe some freedom depending on the $a$--part in the block description \eqref{blockso} of $\alpha_0$ which we can fix arbitrarily without loss of generality.

Finally, we add a linear map $\alpha_{1}: \fc\to \R^{n*}$ such that the component in $ \R^{n*}$ of normalization condition \eqref{norm} vanishes for $\alpha:=\alpha_{-1}+\alpha_0+\alpha_1+d\iota$. Such $\alpha$ then satisfies all the conditions (1),(2),(3) of the statement.
\end{proof}

Based on the above proposition, we adopt the following terminology.

\begin{def*}
We call a linear map $\alpha: \fk\to \so$ a \emph{conformal extension} (of $(\fk,\fh)$) if it satisfies the conditions (1),(2) of the Proposition \ref{conf_ext}. We call the tensor $\kappa$ from point (3) of the proposition the \emph{curvature of the conformal extension} $\alpha$ and we say that the conformal extension is \emph{normal} if the normalization condition \eqref{norm} is satisfied.
\end{def*}

For a conformal extension $\alpha$ of $(\fk,\fh)$, the restriction of $\alpha$ to a map $\alpha_{-1}: \fk/\fh\to \R^n$ along $\fp$ provides $g_o:=\alpha_{-1}^*\nu_{p,q} \in \bigotimes^2 \fk/\fh^*$. If $K,H$ are  Lie groups  with Lie algebras $\fk,\fh$ and $H$ is a closed subgroup of $K$, then the component of identity in $H$ preserves $g_o$ up to a multiple. Thus if also the other connected components in $H$ preserve $g_o$ up to a multiple, then $K/H$ carries a homogeneous conformal structure associated with the conformal extension $\alpha$.

\subsection{Conformal extension for G\" odel metric}\label{ext_godel}
Let us consider conformal class given by the famous G\" odel metrics, \cite{Godel}, on a manifold $M$ with coordinates $(t,x,y,z)$ that is either $\R^4$ or $S^1\times \R^3$ with $t {\rm \ mod\ } 4\sqrt{2}\pi$. We pick the following representative G\" odel metric
$$g=-(dt+e^x dy)^2+dx^2+{\frac12 }e^{2x}dy^2+dz^2.$$
The Lie algebra of  conformal Killing fields of $[g]$ is generated by the time translation $\partial_t$, two space translations $\partial_y, \partial_z$ and further two vector fields $\partial_x-y\partial_y$ and $ (2-2e^{-x})\partial_t+y\partial_x-(1+\frac{y^2}{2}-e^{-2x})\partial_y.$ The flows of the first four vector fields act transitively on both  $\R^4$ and $S^1\times \R^3$. The last vector field vanishes at the origin $o=(0,0,0,0)$ and is complete only on $S^1\times \R^3$.

Let us fix the following Killing fields
$$
k_1:=\partial_z+\partial_t, \ \ k_2:=\partial_x-y\partial_y, \ \ k_3:=\sqrt{2}(\partial_y-\partial_t), \ \ k_4:=-\frac12\partial_t +\frac12 \partial_z,
$$
because these induce a linear isomorphism $\alpha_{-1}: T_oM=\fk/\fh\to \R^4$ such that $$g_o=\alpha_{-1}^*\nu_{1,3}=\alpha_{-1}^*\left[ \begin{smallmatrix}
0& 0 &0&1 \\ 
0&1& 0&0 \\
0&0&1&0 \\
1 &0& 0&0
\end{smallmatrix} \right].$$
Further, we choose $h_1:=(2-2e^{-x})\partial_t+y\partial_x-(1+\frac{y^2}{2}-e^{-2x})\partial_y$ in the stabilizer which allows us to parametrize $\fk$ as $\fk=x_1k_1+x_2k_2+x_3k_3+x_4k_4+x_5h_1$, where the Lie bracket in $\fk$ is {\it minus} the Lie bracket of the corresponding vector fields, i.e., 
$$
[k_2,k_3]=\frac{1}{\sqrt{2}}k_1+k_3-\sqrt{2}k_4,\ \  [k_2, h_1]= -h_1-\sqrt{2}k_3,\ \  [k_3, h_1] = \sqrt{2}k_2.
$$

\begin{lem} \label{ext-godel}
The normal conformal extension $\alpha:\fk \to \sodc$ corresponding to the G\"odel metric according to Proposition \ref{conf_ext} is $\alpha(x_1,x_2,x_3,x_4,x_5)=$
$$
 \left[ \begin {smallmatrix} 0&-\frac{1}{2}{ x_1}+\frac{1}{6}{ x_4}&-\frac{1}{12}{
 x_2}&-\frac{1}{12}{ x_3}&\frac{1}{6}{ x_1}-\frac{1}{8}{ x_4}&0
\\ 
{ x_1}&0&\frac{\sqrt {2}}{4}{ x_3}&-\frac{\sqrt {2}}{4}
{ x_2}&0&-\frac{1}{6}{ x_1}+\frac{1}{8}{ x_4}\\ { x_2}&\frac{\sqrt {2}}{2}{ x_3}&0&\frac{\sqrt {2}}{2}{ x_1}-{ x_3}-\frac{\sqrt {2}}{4}
{ x_4}+\sqrt {2}{ x_5}&-\frac{\sqrt {2}}{4}{ x_3}&\frac{1}{12}{ x_2}
\\ 
{ x_3}&-\frac{\sqrt {2}}{2}{ x_2}&-\frac{\sqrt {2}}{2}{
 x_1}+{ x_3}+\frac{\sqrt {2}}{4}{ x_4}-\sqrt {2}{ x_5}&0&\frac{\sqrt {2}}{4}
{ x_2}&\frac{1}{12}{ x_3}\\ { x_4}&0&-\frac{1}{2}
\sqrt {2}{ x_3}&\frac{\sqrt {2}}{2}{ x_2}&0&\frac{1}{2}{ x_1}-\frac{1}{6}{ x_4}
\\ 
0&-{ x_4}&-{ x_2}&-{ x_3}&-{ x_1}&0
\end {smallmatrix} \right] 
$$
with the curvature $\kappa(\alpha(x_1,x_2,x_3,x_4,x_5),\alpha(y_1,y_2,y_3,y_4,y_5))=$
$$
 \left[ \begin {smallmatrix} 
0&-\frac{\sqrt {2}}{2} z_{23} &
- \frac{\sqrt {2}}{4}z_{13}
-\frac{\sqrt {2}}{8}z_{34}
&\frac{\sqrt{2}}{4}z_{12}+\frac{\sqrt{2}}{8}z_{24} &\frac{\sqrt {2}}{4} z_{23} &0
\\ 
0&\frac13 z_{14}&-\frac{1}{6}z_{12}&
-\frac{1}{6}z_{13}&0&-\frac{\sqrt {2}}{4}
 z_{23}
\\
0&-\frac{1}{6}z_{24}&0&\frac13z_{23}&\frac{1}{6}z_{12}&\frac{\sqrt {2}}{8}z_{34}+\frac{\sqrt {2}}{4}
z_{13}
\\ 
0&-\frac{1}{6}z_{34}&\frac13z_{23}&0&\frac{1}{6}z_{13}&-\frac{\sqrt {2}}{8}z_{24}-\frac{\sqrt {2}}{4}z_{12}
\\ 
0&0&\frac{1}{6}z_{24}&\frac{1}{6}z_{34}&-\frac13z_{14}&\frac{\sqrt {2}}{2} z_{23}
\\ 
0&0&0&0&0&0
\end {smallmatrix} \right],
$$
where we write $z_{ij}=x_iy_j-x_jy_i.$
\end{lem}
\begin{proof}
Elements $k_1,k_2,k_3,k_4$ span a complement $\fc$ of $\fh$ in $\fk$ and $\alpha_{-1}(\sum_{i=1}^4 x_ik_i)=(x_1,x_2,x_3,x_4)^t\in \R^4$. This induces the given $d\iota_0(x_5h_1)$. Then, as in the proof of Proposition \ref{conf_ext}, we compute that $d\iota=d\iota_0$ (because the decomposition $\fc\oplus \fh$ is reductive) and compute the given $\alpha_0,\alpha_1,$ where we choose    $a=0$ for the $a$--part in \eqref{blockso} of $\alpha_0$. The curvature is computed directly by definition.
\end{proof}

\subsection{Homogeneous conformal structures in local coordinates} \label{loc_coord}

 We would like to employ suitable exponential coordinates for our computations and therefore, we shall consider a decomposition of the Lie algebra $\fk$ to subalgebras where the exponential map can be easily computed. In general, one can always consider
 the Levi decomposition of $\fk$ and the Iwasawa decomposition of a maximal semisimple subalgebra of $\fk$, \cite[Section 2.3.5]{parabook}. We group these decompositions into a decomposition $\fk=\fl\oplus  \fa \oplus  \fn$, where

\begin{enumerate}
\item $\fl$ is a maximal compact subalgebra of a maximal semisimple subalgebra of $\fk$,
\item $\fa$ is a maximal diagonalizable subalgebra  (over real numbers) of $\fk$,
\item $\fn$ is the sum of the remaining nilpotent part of the Iwasawa decomposition of a maximal semisimple subalgebra of $\fk$ and the remaining part of the radical of $\fk$ that is not diagonalizable subalgebra  (over real numbers).
\end{enumerate}
Then, we can find (using simple linear algebra and root space decomposition of the maximal semisimple subalgebra of $\fk$) a complement  $\fc$ of $\fh$ in $\fk$ with the following properties

\begin{enumerate}
\item[C1] $\fc=(\fc\cap \fl) \oplus (\fc\cap \fa)\oplus (\fc\cap \fn)$ with dimension $\fc\cap \fl$ minimal possible in the case we are interested in local properties or with dimension $\fc\cap \fl$ maximal possible in the case we are interested in global properties,
\item[C2] each element $\fc\cap \fl$ can be written as a sum $L_{1}+\dots +L_{j}$ for some basis $L_i$ of $\fl$ such that $\exp(L_i)$ can be \emph{easily} computed.
\end{enumerate}

\begin{def*}
We say that ${\sf c}: \fc\to M$ are \emph{exponential coordinates compatible with the decomposition} $\fk= \fl\oplus  \fa \oplus  \fn$ if $\fc$ is a complement of $\fh$ in $\fk$ satisfying conditions C1 and C2, and ${\sf c}$ is defined by
\[
{\sf c}(X) := \exp(L_{1})\dots \exp(L_{j})\exp(X_{\fa})\exp(X_{\fn})o
\]
for $X=L_{1}+\dots +L_{j}+X_{\fa}+X_{\fn}\in (\fc\cap \fl) \oplus (\fc\cap \fa)\oplus (\fc\cap \fn)$. We denote by $\tilde{\sf c}: \fc\to K$ the corresponding natural lift 
$$\tilde{\sf c}(X)=\exp(L_{1})\dots \exp(L_{j})\exp(X_{\fa})\exp(X_{\fn}).$$
\end{def*}

As the first application, these exponential coordinates allow us to (locally) construct the conformal class with prescribed associated conformal extension.

\begin{prop}\label{c-coframe}
Let $\alpha: \fk\to \so$ be a conformal extension of $(\fk,\fh)$ and let $\fc$ be a complement of $\fh$ in $\fk$ satisfying conditions C1, C2. If $e^*:=(e^1, \dots, e^n): T\fc\to \R^n$ is a coframe given by $\alpha_{-1}\circ \tilde{\sf c}^*\omega_K$ for the Maurer--Cartan form $\omega_K$ of some Lie group $K$ with the Lie algebra $\fk$, then the conformal class $[e^*\nu_{p,q}(e^*)^t]$ on $\fc$ is locally homogeneous conformal geometry with associated conformal extension $\alpha$.
\end{prop}
\begin{proof}
Under the assumptions of the proposition, we have the lift $\tilde{\sf c}: \fc\to K$ and thus we can pullback the Maurer--Cartan form $\omega_K: TK\to \fk$ on $K$ to $\tilde{\sf c}^*\omega_K: T\fc\to \fk$. So after composition with $\alpha$ we get a map  $T\fc\to \so$ and thus $\alpha_{-1}\circ \tilde{\sf c}^*\omega_K$ is a coframe $e^*:=(e^1, \dots, e^n): T\fc\to \R^n$. This defines the conformal class $[e^*\nu_{p,q}(e^*)^t]$ on $\fc$, which by construction is locally homogeneous with associated conformal extension $\alpha$. Indeed, if we assume that ${\sf c}: \fc\to K/H$ are exponential coordinates compatible with the decomposition $\fk= \fl\oplus  \fa \oplus  \fn$ for suitable closed subgroup $H$ of $K$ with Lie algebra $\fh$, then $[e^*\nu_{p,q}(e^*)^t]$ is the description of the invariant conformal class induced by the element $\alpha_{-1}^*\nu_{p,q}\in \bigodot^2 \fk/\fh^*$ in these local coordinates.
\end{proof}

It will be useful later to adopt the following definition.

\begin{def*}
We call the coframe $e^*$ from Proposition \ref{c-coframe} a \emph{$\fc$--coframe} and we say that its dual frame is a \emph{$\fc$--frame}.
\end{def*}

\subsection{Exponential coordinates for G\" odel metrics} \label{coord-godel}
Let us illustrate ideas of Section \ref{loc_coord} on the example from Section \ref{ext_godel}.

\begin{lem}\label{lem-coord-godel}
Let us consider the complement $\fc=\langle k_1,k_2,k_3,k_4\rangle$ for the conformal class of the G\"odel metrics from Section \ref{ext_godel}. Then there exists a decomposition $\fk=\fl\oplus \fa\oplus \fn$ such that
\begin{gather*}
e_1=\partial_{a_1}+\partial_{a_2},\  
e_2=\frac12(\partial_{a_3}+2a_4\partial_{a_4}) ,\ 
e_3=\sqrt{2}(-\partial_{a_2}+\partial_{a_4}),
e_4=\frac12(\partial_{a_1}-\partial_{a_2})\\
e^1=\frac12(da_1+da_2-2a_4da_3+da_4),\ e^2=2da_3,\ e^3= \frac1{\sqrt{2}}(-2a_4da_3+da_4),\\
 e^4=da_1-da_2+2a_4da_3-da_4
\end{gather*}
form the corresponding $\fc$--(co)frames in the exponential coordinates $$\fc=(a_1,a_2,a_3,a_4)\to M$$ compatible with the decomposition $\fl\oplus  \fa \oplus  \fn$. In particular,
\begin{gather*}
g=da_1^2-da_2^2+2a_4(da_2da_3+da_3da_2)-(da_2da_4+da_4da_2)
\\ +(4-2a_4^2)da_3^2
+a_4(da_3da_4+da_4da_3)-\frac12da_4^2.
\end{gather*}
\end{lem}
\begin{proof}
Let us start with the observation that the parametrization $$a_1(\frac12 k_1+k_4)+a_2(\frac12k_1-k_4)+2a_3k_2+a_4(\frac12k_1 +\frac{1}{\sqrt{2}}k_3-k_4)+a_5h_1$$ identifies $\fk$ with the following matrix Lie algebra
$$
\left[ \begin{smallmatrix}
 a_1 &0&0 \\ 0&a_2+a_3+2a_5&\frac12 a_5 \\
 0 & a_4-a_5&a_2-a_3+2a_5
\end{smallmatrix} \right]\cong \R\oplus \frak{gl}(2,\R),
$$
where $\fl=\frak{so}(2)$, $\fa$ is the diagonal part and $\fn$ is the strictly lower diagonal part of this matrix Lie algebra.
Thus we find a complement $\fc$ parametrized by $(a_1,a_2,a_3,a_4)$ satisfying the conditions C1, C2, because $\fc\cap \fl=0$, $\fc\cap \fa=\langle a_1,a_2,a_3\rangle$ and $\fc\cap \fn=\langle a_4\rangle$.  Then from the parametrization we obtain
\[
\alpha_{-1}:\fc\to \R^4, \ \  (a_1,a_2,a_3,a_4)\mapsto  T(a_1,a_2,a_3,a_4)^t,
\]
where
\[T=\left[ \begin{smallmatrix}
\frac12&\frac12&0& \frac12  \\ 
0&0&2& 0  \\ 
0&0&0& \frac1{\sqrt{2}} \\ 
1&-1&0& -1  \\ 
\end{smallmatrix} \right].\]
Thus 
\[{\sf c}(a_1,a_2,a_3,a_4)=\exp\left[ \begin{smallmatrix}
 a_1 &0&0 \\ 0&a_2+a_3&0 \\
 0 & 0&a_2-a_3
\end{smallmatrix} \right] \exp\left[ \begin{smallmatrix}
0 &0&0 \\ 0&0&0\\
 0 & a_4&0
\end{smallmatrix} \right] o
\] are the corresponding exponential coordinates compatible with the decomposition $\fk= \fl\oplus  \fa \oplus  \fn$. Consequently,
\[
\tilde{\sf c}^*\omega_K=\left[ \begin{smallmatrix}
da_1&0&0 \\ 
0&da_2+da_3&0 \\
 0 &-2a_4da_3+da_4&da_2-da_3
\end{smallmatrix} \right]
\]
can be restricted to the coframe $(da_1,da_2,da_3,-2a_4da_3+da_4)$ valued in $\fc$. Thus we obtain $\fc$--coframe $e^*=(da_1,da_2,da_3,-2a_4da_3+da_4)T^t$, i.e.,
$$g=(da_1,da_2,da_3,-2a_4da_3+da_4)T^t\nu_{1,3} T(da_1,da_2,da_3,-2a_4da_3+da_4)^t$$
 according to Proposition \ref{c-coframe}.
\end{proof}

\subsection{Homogeneous conformal Cartan geometries}\label{cartan}

The key objects for the tractorial approach to first order BGG operators are Cartan geometries of type $(G,P)$ for Lie groups $G$ with the Lie algebra $\so$ and their parabolic subgroups $P$ with the Lie algebra $\fp$. Let us recall that a Cartan geometry of type $(G,P)$ consists of a principal $P$--bundle $\ba$ over $M$ together with a Cartan connection $\omega: T\ba\to \fg$, i.e., a $P$--equivariant absolute parallelism $\omega$ that reproduces fundamental vector fields of the right $P$--action on $\ba$, \cite[Sections 1.5 and 1.6]{parabook}.

Let $PO(p+1,q+1)$ be the projectivization of $O(p+1,q+1)$, the orthogonal Lie group preserving $\bg$ on $\mathbb{T}$, and $P$ the stabilizer of a line generated by the first vector of standard basis of $\mathbb{T}$. There always is a Cartan geometry of type $(PO(p+1,q+1),P)$ that solves the equivalence problem of conformal geometries, \cite[Theorem 1.6.7]{parabook}. However, the representation $\rho$ of $\so$ does not have to integrate to a representation of $PO(p+1,q+1)$ and thus a global tractorial approach requires different choices of $G$, namely $O(p+1,q+1)$, $SO(p+1,q+1)$, $SO_o(p+1,q+1)$ or $Spin(p+1,q+1)$ to where $\rho$ integrates. However, the existence of Cartan geometries of type $(G,P)$ requires the conformal geometries to satisfy certain topological obstructions. For example, for $G=Spin(p+1,q+1)$ the conformal geometry has to be spin.

In the case of $K$--homogeneous conformal geometries on $M=K/H$, \cite[Theorem 1.5.15]{parabook} relates topological obstructions to the integrability of the restriction $d\iota$ of a conformal extension $\alpha: \fk\to \so$ to $\fh$ onto a Lie group homomorphism $\iota: H\to P$ such that $\Ad(\iota(h))\circ \alpha=\alpha\circ \Ad(h)$ for all $h\in H$. The Theorems \cite[Theorem 1.5.15 and 3.1.12]{parabook} then imply the following statement.

\begin{prop} \label{gen-ext}
The pair $(\alpha,\iota)$ provides the Cartan geometry $(\ba, \om)$ of type $(G,P)$ as follows

\begin{enumerate}
\item $\ba:=K\times_{\iota(H)}P,$ and
\item $\omega:=\omega_\alpha$, where $\omega_\alpha$ is the unique Cartan connection with the property $j^*\om_\alpha=\alpha\circ \om_K$ for the natural inclusion $j: K\to K\times_{\iota(H)}P$ and the Maurer--Cartan form $\omega_K$ on $K$.
\item The Cartan geometry is normal if and only if the conformal extension is normal. 
\end{enumerate}
\end{prop}

Let us describe the normal Cartan geometry in local coordinates ${\sf c}: \fc\to M=K/H$. Let us emphasize that locally, this construction can be done under the assumptions of Proposition \ref{c-coframe} with the $\fc$--coframe $e^*=(e^1, \dots, e^n): T\fc\to \R^n$ as the key ingredient. 

Firstly, the Cartan bundle $\ba=K\times_{\iota(H)} P$ is locally trivialized to $\fc \times P$ via the identification $$(X,p)\mapsto \tilde{\sf c}(X)up$$ for $u$ in the fiber over $o=eH$ and we denote the corresponding natural section by $$\s^u: \fc \mapsto \fc \times P.$$ 
Then the pullback $(\s^u)^*\om$ of the Cartan connection $\om=\om_\alpha$ to $\fc$ is the matrix of one--forms on $\fc$ as follows
\begin{align} \label{Cartan_connection} \tag{Con}
\left[ \begin{smallmatrix}
 a_ke^k & b_ke^k & c_ke^k & d_ke^k & 0 \\
 e^i & A_{k}e^k & E_{k}e^k & F_{k}e^k & * \\
 e^{j+p} & C_{k}e^k & B_{k}e^k & * & * \\
 e^{i+q} & D_{k}e^k & * & * & * \\
 0 & * & * & * & *
\end{smallmatrix} \right],
\end{align}
where $i=1, \dots, p$ and $j=1, \dots, q-p$ and $k=1,\dots,p+q$. Here blocks are one--forms valued in parts of \eqref{blockso}.
\begin{itemize}
\item  Components $(a_k,A_k,B_k,C_k,D_k,E_k,F_k)e^k$ of \eqref{Cartan_connection} are one--forms valued in $\frak{co}(p,q)$ and provide a connection form of a Weyl connection on $\fc$. Thus they are determined up to a choice of the Weyl connection by vanishing of the torsion. 
If $a_k=0$, then it is a Levi-Civita connection for some metric in the conformal class.
 
\item Components $(b_k,c_k,d_k)e^k$ of \eqref{Cartan_connection} are one--forms valued in $\R^{n*}$ and provide the Rho--tensor $\Rho$ of the corresponding Weyl connection on $\fc$. In particular, they are uniquely determined by the one--forms from the above point.
 \end{itemize}
 
Finally, the pullback $(\s^u)^*\kappa$ of the curvature of the Cartan geometry to $\fc$ is a matrix of $2$--forms on $\fc$ as follows
\begin{align} \label{Cartan_curvature} \tag{Cur}
\left[ \begin{smallmatrix}
 0 & Y_{kl}e^k  \wedge e^l & 0 \\
 0 & W_{kl}e^k \wedge e^l & * \\
 0 &  0 & 0
\end{smallmatrix} \right],
\end{align}
where $l=1,\dots p+q$.
The component $W$  taking values in $\wedge^2 \R^{n*} \otimes \frak{so}(p,q)$ corresponds to the Weyl curvature of the Weyl connections. The component $Y$ taking values in $\wedge^2 \R^{n*} \otimes \R^{n*}$ corresponds to the Cotton--York tensor $Y$ of the Weyl connection from above.

Let us relate the components of \eqref{Cartan_connection} with the map $\alpha$.

\begin{prop}\label{localext}
Let $(e^1, \dots, e^n)$ be a $\fc$--coframe and $u=(0,\id)\in \fc\times P$. 
There exists normal conformal extension $\alpha: \fk \to \so$ of $(\fk,\fh)$ such that the pullback $(\s^u)^*\om$ to $T\fc$ of the Cartan connection \eqref{Cartan_connection} satisfies \[(\s^u)^*\om=\alpha \circ \big((\alpha_{-1})^{-1}\circ (e^1, \dots, e^n)+\sum_i H_ie^i\big)\] 
for certain functions $H_1,\dots, H_n : \R^n \to \fh$ such that $$\omega_K|_{T\fc}=(\alpha_{-1})^{-1}\circ (e^1, \dots, e^n)+\sum_i H_ie^i\in \fc\oplus \fh.$$
\end{prop}
\begin{proof}
The claim on the existence of $\alpha$ and the pullback is clear from the formula \eqref{Cartan_connection}, because both the Cartan connection $\omega$ and $\omega_K$ are left--invariant. In particular, this recovers the construction of $\alpha$ from $\alpha_{-1}$ from the proof of Proposition \ref{c-coframe}. So it remains to compare $(e^1, \dots, e^n)$ with $\omega_K$, which provides the claimed functions.
\end{proof}

\begin{cor} \label{localext-alg}
If $\fc$ is a complementary Lie subalgebra, then functions $H_1,\dots, H_n$ vanish and the parts of \eqref{Cartan_connection} coincide with the restriction of $\alpha$ to $\fc$ in the $\fc$--coframe.
\end{cor}

Thus if $\fc$ is a subalgebra of $\fk$, then $\alpha|_{\fc}$ can be directly observed from $(\s^u)^*\om$ and $\alpha|_{\fh}$ is uniquely determined by the conditions $\alpha\circ \Ad(h)=\Ad(\iota(h)) \circ \alpha$ and  $\alpha|_\fh=d\iota$. In general, one needs to compute the maps $H_i:\fc \to \fh$ to observe $\alpha$ from $(\s^u)^*\om$. 

\subsection{Cartan connection for G\" odel metrics.} \label{cc-godel}
In the case $M=\R^4$, we have a solvable Lie group $C$ with the Lie algebra $\fc$ acting simply transitively on $M$. In the case $M=S^1\times \R^3$, we have a Lie group $K=\R\times SO(2)\times Sl(2,\R)$ acting transitively on $M$ with stabilizer $H=\Delta(SO(2))$, that is the diagonal in product of $SO(2)=\R/4\sqrt{2}\pi\mathbb{Z}$ and the maximal compact subgroup of $Sl(2,\R)$. In both cases, we can use the conformal extension $\alpha$ from Lemma \ref{ext-godel} to describe the Cartan geometry using Proposition \ref{localext} and Corollary \ref{localext-alg}.

\begin{lem} \label{lem-cartan-godel}
On both $M=\R^4$ and $M=S^1\times \R^3$ with the conformal class of the G\"odel metrics, we have $\fc$--frame and $\fc$--coframe as follows
\begin{gather*}
e_1:=\partial_t+\partial_z,\ \ e_2:=\partial_x, \ \ e_3:=\sqrt{2}(e^{-x}\partial_y-\partial_t), \ \ e_4:=-\frac12\partial_t+\frac12\partial_z, \\
e^1:=\frac12 (dt +e^x dy+dz), \ \ e^2:=dx, \ \ e^3:=\frac{e^x}{\sqrt{2}} dy, \ \  e^4:=-dt-e^xdy+dz,
\end{gather*}
and the pullback $(\s^u)^*\om$ of the corresponding Cartan connection $\om$ from Proposition \ref{localext} takes form
$$
\left[ \begin {smallmatrix} 
0&-{\frac{1}{2}}{ e^1}+{\frac{1}{6}}{ e^4}&-{\frac{1}{12}}{ e^2}&-{\frac{1}{12}}{ e^3}&{\frac{1}{6}}{ e^1}-{\frac{1}{8}}{ e^4}&0\\ 
e^1&0&{\frac{\sqrt {2}}{4}}{ e^3}&-{\frac{\sqrt {2}}{4}}{ e^2}&0&-{\frac{1}{6}}{ e^1}+{\frac{1}{8}}{ e^4}\\
e^2&{\frac{\sqrt {2}}{2}}{ e^3}&0&{\frac{\sqrt {2}}{2}}{ e^1}-{ e^3}-{\frac{\sqrt {2}}{4}}{ e^4}&-{\frac{\sqrt {2}}{4}}{ e^3}&{\frac{1}{12}}{ e^2}\\
e^3&-{\frac{\sqrt {2}}{2}}{ e^2}&-{\frac{\sqrt {2}}{2}}{ e^1}+{ e^3}+{\frac{\sqrt {2}}{4}}{ e^4}&0&{\frac{\sqrt {2}}{4}}{ e^2}&{\frac{1}{12}}{ e^3}\\
e^4&0&-{\frac{\sqrt {2}}{2}}{ e^3}&{\frac{\sqrt {2}}{2}}{ e^2}&0&{\frac{1}{2}}{ e^1}-{\frac{1}{6}}{ e^4}\\
0&-e^4&-e^2&-e^3&-e^1&0\end {smallmatrix} \right] 
.
$$
In particular, the component of $(\s^u)^*\om$ in $\frak{co}(p,q)$ corresponds to the Levi--Civita connection of the G\"odel metric $g$ with Christoffel symbols
\[
\Gamma_{tx}^t= \Gamma_{xt}^t=1,\ \Gamma_{xy}^t=\Gamma_{yx}^t=\Gamma_{ty}^x=\Gamma_{yt}^x=\frac{e^x}2,\ \Gamma_{yy}^x=\frac{e^{2x}}2,
 \Gamma_{tx}^y=\Gamma_{xt}^y=\frac{-1}{e^x}
\]
and the component of $(\s^u)^*\om$ in $\R^{n*}$ corresponds to the $\Rho$--tensor
\[
-\frac{1}{24}(10dt^2 +10e^x(dtdy+dydt) +2dx^2+11e^{2x}dy^2+2dz^2).
\]
Moreover, the curvature of the Cartan connection $\omega$ is given by $\kappa$ from Lemma \ref{ext-godel} viewed as a constant function $M\to \wedge^2 \R^{n*}\otimes \sodc.$
\end{lem}
\begin{proof}
The conformal Killing fields corresponding to the complement $\fc$ are para\-met\-rized by $(a_1,a_2,a_3,a_4)$ in the following way
\[
(\fc\cap \fa)\oplus (\fc\cap \fn)=\{a_1\partial_z +a_2 \partial_t+a_3 (2\partial_x-2y\partial_y)\}\oplus \{a_4\partial_y\}.
\]
Since the product of exponential maps corresponds to the composition of flows of the conformal Killing fields, we can compute that the transition from $(a_1,a_2,a_3,a_4)$--coordinates to $(t,x,y,z)$--coordinates on $M$ takes the form
\[
(a_1,a_2,a_3,a_4)\mapsto (a_2,  2a_3, a_4e^{-2a_3}, a_1).
\]

Thus the $\fc$--(co)frame is obtained from Lemma \ref{ext-godel} using the transition. Since $\fc$ is a subalgebra of $\fk$, the rest is then obtained according to Proposition \ref{localext} and Corollary \ref{localext-alg} using $\alpha$ and the $\fc$--(co)frame. Let us note that these are non--holonomic frames and thus there is a contribution of derivatives of vectors in the $\fc$--frame to the Christoffel symbols of the corresponding Weyl connection, that is the Levi--Civita for $g$ in this case. 
\end{proof}

\section{First BGG operators on homogeneous conformal geometries} \label{sec3}

\subsection{On tractorial approach to first BGG operators}

Let us summarize some details from \cite{H-srni} that we will need to prove the Theorem \ref{thm-intro}. The basic idea is to prolong the first BGG operator from a bundle $\mathcal{X}\to M$ with the standard fiber $\mathbb{X}$ to a linear (prolongation) connection on the tractor bundle $\mathcal{V}\to M$ with standard fiber $\mathbb{V}$, where  $\mathbb{X}$ can be identified with the $\frak{co}(p,q)$--module of the lowest weight in the representation $\rho: \so\to \frak{gl}(\mathbb{V}).$

The result of \cite{new-norm} states that the prolongation connection can be constructed from the representation $\rho$ using the conformal Cartan connection and that solutions of the first BGG operator are in bijective correspondence with parallel sections of the prolongation connection. Since on homogeneous conformal geometries, the prolongation connection is an invariant connection and the conformal Cartan connection is completely described as in Proposition \ref{gen-ext} by the conformal extension $\alpha$, the tractorial approach provides the correct setting for the computation of solutions of the first BGG operators. Before we start proving the Theorem \ref{thm-intro}, let us recall the representations $\rho$ corresponding to the most studied first BGG operators.
\noindent \\[2mm] {\bf (1)}
 The conformal class $[g]$ can be viewed as a section of a trivial line subbundle of $\bigodot^2T^*M[2]$ representing the inclusion $\mathcal{E}[-2] \hookrightarrow \bigodot^2T^*M$ provided by the conformal class.
The standard fiber of this line bundle is $\mathbb{X}=\R$. 
The corresponding tractor bundle has the fiber $\mathbb{V}=\R$ for trivial representation of $\rho$ and the section $[g]$ defines a conformal metric $\bg$ on the standard tractor bundle with the standard fiber $\mathbb{T}$. These allow to rise and lower indices at the price of adding the conformal density. 
\noindent \\[2mm]  {\bf (2)}
(Almost) Einstein scales are sections $\sigma$ of bundle $\mathcal{E}[1]$ such that $\sigma^{-2}\bg$ are Einstein metrics in $[g]$ (on open dense subsets on $M$ where sections $\sigma$ are non--vanishing), \cite{brinkman,gov,G-einstein}. The standard fiber is $\mathbb{X}=\R[1]$ in this case and the corresponding tractor bundle has fiber $\mathbb{V}=\mathbb{T}$. We show later on examples that the zero locus of $\sigma$ inherits a distinguished geometric structure.
\noindent \\[2mm]  {\bf (3)}
 Twistor spinors are sections of bundles where standard fibers $\mathbb{X}$ are tensor products of the spinor representations of $\frak{so}(p,q)$ with $\mathcal{E}[\frac12]$ satisfying the twistor equation, \cite{BTGK,bar,PR}. The standard fibers $\mathbb{V}$ of the corresponding tractor bundles are the spinor representations $\mathbb{D},\mathbb{D}^\pm$.
\noindent \\[2mm]  {\bf (4)}
 Conformal Killing vectors are sections of $TM$ whose Lie derivative preserves the conformal class and thus $\mathbb{X}=\R^n$. 
The corresponding tractor bundle has the standard fiber $\mathbb{V}=\so\cong \wedge^2 \mathbb{T}$.
\noindent \\[2mm]  {\bf (5)}
 Conformal Killing $k$--tensors are solutions of first BGG operators on $\bigodot^k_0 TM\cong \bigodot^k_0 T^*M[2k]$, \cite{black-holes} and references therein, and $\mathbb{X}=\bigodot^k_0 \R^n$. The corresponding tractor bundle $\mathbb{V}$ has the standard fiber $\boxtimes(\bigodot^k\so)$.
\noindent \\[2mm]  {\bf (6)}
 Conformal Killing--Yano $(k-1)$--forms are solutions of first BGG operators on $\wedge^{k-1} T^*M[k]$, \cite{black-holes,Kress,Y}, and $\mathbb{X}=\wedge^{k-1}\R^n[2-k]$. The corresponding tractor bundle $\mathbb{V}$ has the standard fiber $\wedge^{k}\mathbb{T}$.

\subsection{Local solutions of first BGG operators on homogeneous conformal geometries}

In this section, we work with the Cartan connections of type $(G,P)$ for a Lie group $G$ with Lie algebra $\so$ such that $\rho$ integrates to a representation  $\lambda$ of $G$. We assume $K,H$ and $\iota$ are such that we can use the description of the Cartan connections of type $(G,P)$ associated with the conformal extension $\alpha: \fk\to \so$ of $(\fk,\fh)$ from Section \ref{cartan}. In general, this cannot be done globally. Nevertheless, the local description of this Cartan geometry is always available using the $\fc$--(co)frame. Thus for the claim of the following local version of Theorem \ref{thm-intro}, we can consider this assumption without loss of generality. The Theorem \ref{thm-intro} then directly follows from this theorem.

\begin{thm} \label{thm-loc}
Let $(M,[g])$ be a homogeneous conformal geometry with associated conformal extension $\alpha: \fk\to \so$ of $(\fk,\fh)$. Let $\rho: \fg\to \frak{gl}(\mathbb{V})$ be  representation encoding first BGG operator on bundle $\mathcal{X}\to M$. Then there are representations $\Phi: \fk\to \frak{gl}(\mathbb{S})$ and $d\mu: \fh\to \frak{gl}(\mathbb{X})$ and an $\fh$--equivariant projection $\pi: \mathbb{S}\to \mathbb{X}$ describing the (local) solutions of the first BGG operator.

In particular, there is an inclusion $\mathbb{S}\subset \mathbb{V}$ such that the function $$\exp(X)\mapsto \exp(-\Phi(X))(s)\in \mathbb{V}$$ for $s\in \mathbb{S}$ and $X$ in some neighborhood of $0$ in $\fk$ induces a natural prolongation of the (local) solution of the first BGG operator to a section of a tractor bundle $\mathcal{V}$ over $M$ with standard fiber $\mathbb{V}$ that is parallel for prolongation connection on $\mathcal{V}$.
\end{thm}
\begin{proof}
Firstly, let us consider the tractor bundle
$$\mathcal{V}:=K\times_{\lambda\circ \iota(H)}\mathbb{V} \to K/H$$
and interpret its sections as  $H$--equivariant functions $s: K\to \mathbb{V}$ for the right multiplication on $K$ and the action $\lambda\circ \iota$ on $\mathbb{V}$. Let us use the notation $s\in \Gamma(\mathcal{V})^\ell$ for $H$--equivariant function $s: K\to \mathbb{V}^\ell$. This allows us to define fundamental derivative $D^\fk_t s:=\om_K^{-1}(t).s$
 for $H$--equivariant function $t: K\to \fk$ and $s\in  \Gamma(\mathcal{V})$, where $.$ is the directional derivative in the direction of  the vector field $\om_K^{-1}(t)$ on $K$.  Fundamental derivative $D^\fk$ provides a uniform description of $K$--invariant linear connections $ \Omega^0(\mathbb{V})\to \Omega^1(\mathbb{V})$, \cite{parabook}, where we use notation
\begin{align*}
\Omega^k(\mathbb{V}):=\Gamma(K\times_{(\Ad^k\otimes \lambda)\circ \iota(H)}\wedge^k\R^{n*}\otimes \mathbb{V})
\end{align*}
for the spaces of $\mathbb{V}$--valued k--forms on $K/H$. Precisely, each $K$--invariant linear connection is given as 
$$\nabla^\Phi:=D^\fk+\Phi,$$
where $\Phi: \fk\to \frak{gl}(\mathbb{V})$ is an $H$--equivariant map satisfying $\Phi(Y)=d\lambda\circ \alpha(Y)$ for all $Y\in \fh$. Its curvature $R^{\Phi}: K\to  \wedge^2\R^{n*}\otimes \frak{gl}(\mathbb{V})$ is given for $X_0,X_1\in \fk$ as
$$R^{\Phi}(\alpha(X_0)+\fp,\alpha(X_1)+\fp)=[\Phi(X_0),\Phi(X_1)]-\Phi([X_0,X_1]).$$
In particular, since the representation $\rho$ satisfies $\rho\circ \alpha(Y)=d\lambda\circ \alpha(Y)$ for all $Y\in \fh$, there is a $K$--invariant linear connection
\begin{align*}
\nabla^{\rho\circ \alpha}=D^\fk+\rho\circ \alpha
\end{align*}
on $\mathcal{V}$ that is the usual tractor connection with curvature $R^{\rho\circ \alpha}=\rho\circ \kappa.$

Moreover, there is the Kostant's codifferential
$\partial^*: \Omega^{k+1}(\mathbb{V})\to \Omega^{k}(\mathbb{V})$  defined pointwise via $\partial^*: \wedge^{k+1}\R^{n*}\otimes \mathbb{V}\to \wedge^{k}\R^{n*}\otimes \mathbb{V}$ as
 \begin{align*}
 \partial^*(Z_0&\wedge \dots  \wedge Z_k\otimes v)=\sum_{j} (-1)^{j+1}Z_0\wedge \dots  \wedge \hat Z_j\wedge \dots \wedge Z_k\otimes \rho(Z_j)(v),
 \end{align*}
and we denote by $\pi_i$ projections onto the cohomology spaces $$\mathcal{H}^i(\mathbb{V})=Ker(\partial^*)/Im(\partial^*).$$ In particular, $\mathcal{X}=\mathcal{H}^0(\mathbb{V})$, i.e., it holds $\mathbb{X}=Ker(\partial^*)/Im(\partial^*)=\mathbb{V}/Im(\partial^*)$ pointwise and this induces the representation $d\mu$.

The next ingredient is the splitting operator $\mathcal{L}_0: \mathcal{H}^0(\mathbb{V})\to Ker(\partial^*)= \Omega^{0}(\mathbb{V})$ defined as $\mathcal{L}_0=\id-Q \partial^*\nabla^{\rho\circ \alpha}$ for a particular operator $Q: Ker(\partial^*)\to Ker(\partial^*)$ that is polynomial in $\partial^* \nabla^{\rho\circ \alpha}$ with coefficients determined by the representation theory, \cite{casimir,relative}. Let us note that in the applications, these coefficients can be determined by the property that both $Q \partial^* \nabla^{\rho\circ \alpha}$ and $\partial^* \nabla^{\rho\circ \alpha} Q$ act as identity on $Im(\partial^*).$
Then the operator $$\mathcal{D}:=\pi_1\nabla^{\rho\circ \alpha}\mathcal{L}_0$$ for the tractor connection $\na^{\rho \circ \alpha}$ is the (standard) first BGG operator.

A difference of two $K$--invariant linear connections is given by an $H$--equivariant map $\psi: \fk \to \frak{gl}(\mathbb{V})$ satisfying $\psi(Y)=0$ for all $Y\in \fh$. If $\psi\in (\fk^*\otimes \frak{gl}(\mathbb{V}))^1$, $\psi(s)\in Im(\partial^*)$ for all $H$--equivariant functions $s: K\to \mathbb{V}$, then $\mathcal{D}=\pi_1 \nabla^{\rho\circ \alpha+\psi} \mathcal{L}_0$.
 Moreover, by result of \cite{new-norm,casimir},
there is a unique $\Psi \in (\fk^*\otimes \frak{gl}(\mathbb{V}))^2$ vanishing on $\fh$ such that $\partial^*R^{\rho\circ \alpha+\Psi}(s)=0$
 for all $s\in \mathbb{V}$.  This way we obtain $\Phi:=\rho\circ \alpha+\Psi$ and the corresponding connection $\nabla^{\Phi}$ is the prolongation connection.

The solutions of $\mathcal{D}$ are in bijective correspondence with parallel sections of the invariant connection $\nabla^{\Phi}$ and thus can be algebraically computed, \cite{KoNo}. In particular, one iteratively computes the sets
\begin{align}\label{holannih1} \tag{S0}
\mathbb{S}^0:=\{v\in \mathbb{V}: R^{\Phi}(\alpha(X_0)+\fp,\alpha(X_1)+\fp)(v)=0,\  X_0,X_1\in \fk\}
\end{align}
and
\begin{align}\label{holannih2} \tag{Sk}
\mathbb{S}^k:=\{v\in \mathbb{S}^{k-1}: \Phi(X)(v)\in \mathbb{S}^{k-1},\ X\in \fk\}.
\end{align}

Since $\mathbb{V}$ is finite--dimensional, we get $\mathbb{S}^k=\mathbb{S}^{k+1}=\dots =:\mathbb{S}$ for $k$ large enough as the set of local solutions. By the definition of $\mathbb{S}^0$, the map $\Phi: \fk\to \frak{gl}(\mathbb{V})$ restricts to the claimed representation $\Phi: \fk \to \frak{gl}(\mathbb{S})$. Then the claimed formula extends it locally to a section of $\mathbb{V}$ parallel for the connection $\nabla^{\Phi}$.  The $\fh$--equivariant projection $\pi: \mathbb{S}\to \mathbb{X}$ is induced by the inclusion $\mathbb{S}\subset \mathbb{V}$ and the projection $\pi_0: \mathbb{V}=Ker(\partial^*)\to \mathbb{X}$.
\end{proof}

Let us emphasize that the sets \eqref{holannih1},\eqref{holannih2} provide the infinitesimal holonomy of the connection $\na^\Phi$, \cite{KoNo}.

\subsection{Normal solutions and algorithm for computing all solutions
}\label{section3.3}

There is a special class of solutions of first BGG operators characterized in the homogeneous setting by the following property.

\begin{def*}
We call $s\in \mathbb{S}$ a \emph{normal solution} if $s$ belongs to a subrepresentation $\mathbb{N}$ of $\mathbb{S}$, where $\Phi|_\mathbb{N}=\rho\circ \alpha|_\mathbb{N}$.
\end{def*}

Let us verify that this coincides with the usual definition of normal solutions as parallel sections for the tractor connection.

\begin{prop}
Sections $s$ of the tractor bundle corresponding to normal solutions of first BGG operators are parallel sections for the tractor connection. In particular, they are annihilated by the action of the curvature of the Cartan connection.
\end{prop}
\begin{proof}
Since the tractor connection is an invariant connection $\nabla^{\rho\circ \alpha}$ with the curvature $\rho \circ \kappa$, each parallel section is annihilated by the infinitesimal holonomy and in particular, by the action of the curvature of the Cartan connection. By construction of the prolongation connection, this implies that each parallel section for the tractor connection is a normal solution in $\mathbb{S}$. Conversely, if $s$ corresponds to a normal solution, then the condition $\Phi|_\mathbb{N}=\rho\circ \alpha|_\mathbb{N}$ implies that actions of infinitesimal holonomies of both the prolongation and tractor connection coincide on $\mathbb{N}$. Thus $s$ is a  parallel section for the tractor connection.
\end{proof}

Normal solutions have the following remarkable property. If $\Phi_1|_{\mathbb{N}_1}: \fk\to \frak{gl}(\mathbb{N}_1)$ and $\Phi_2|_{\mathbb{N}_2}: \fk\to \frak{gl}(\mathbb{N}_2)$ describe the normal solutions of first BGG operators corresponding to representations $\rho_1: \so\to \frak{gl}(\mathbb{V}_1)$ and $\rho_2:  \so\to \frak{gl}(\mathbb{V}_2)$, then $\Phi_1|_{\mathbb{N}_1}\otimes \Phi_2|_{\mathbb{N}_2}: \fk\to \frak{gl}(\mathbb{N}_1\otimes \mathbb{N}_2)$ describes the normal solutions of the first BGG operators corresponding to representation $\rho_1\otimes \rho_2:  \so\to \frak{gl}(\mathbb{V}_1\otimes \mathbb{V}_2).$
This is usually referred to as \emph{BGG coupling}, \cite{coupling}, and can be generalized to all the operations available from the theory of $\so$--representations.

Let us now summarize how we compute the representation $\Phi$ and the normal solutions in the practice. 
\noindent \\[2mm] {\bf (1)}
 The starting point is the normal conformal extension $\alpha: \fk\to \so$ of $(\fk,\fh)$ and a representation $\rho: \so\to \frak{gl}(\mathbb{V}).$
\noindent \\[2mm] {\bf (2)}
 We use the action of $\rho(E)$ to determine the grading of $\mathbb{V}$, the projective slot $\mathbb{X}$ and the projection $\pi_0: \mathbb{V}\to \mathbb{X}.$
\noindent \\[2mm] {\bf (3)}
 Then we consider $\Phi=\alpha$ and use the formulas used in definition of sets \eqref{holannih1} and \eqref{holannih2} to compute the infinitesimal holonomy $hol(\alpha)\subset \so$.
\noindent \\[2mm] {\bf (4)}
The normal solutions are elements of $\mathbb{V}$ annihilated by $\rho(hol(\alpha))$.
\noindent \\[2mm] {\bf (5)}
We compute the prolongation connection $\nabla^{\Phi}:=\nabla^{\rho\circ\alpha+\Psi}$ by induction with respect to the irreducible grading components of the map $\Psi$. 
\begin{itemize}
\item We start with $\psi_0=0$. 
\item In the induction step, we compute
$$\psi_k:=\psi_{k-1}-\frac{1}{c_k}q_i\big((\partial^*\otimes \id_{\mathbb{V}^*})R^{\rho\circ \alpha+\psi_{k-1}}\big)$$
for certain integers $c_k$, where $q_i$ denotes the projection to the lowest non--zero homogeneity $i$ that either can be determined by the representation theory, \cite{casimir}, or can be chosen directly so that it kills some component of $(\partial^*\otimes \id_{\mathbb{V}^*}) R^{\rho\circ \alpha+\psi_{k}}$ in the given homogeneity. 
\item Since the image of $\partial^*\otimes \id_{\mathbb{V}^*}$ does not lower the homogeneity and there is only a finite number of irreducible grading components in $\R^{n*}\otimes \frak{gl}(\mathbb{V})$, we get $\Psi:=\psi_k$ in finitely many steps. 
\end{itemize}
\noindent {\bf (6)}
We iteratively compute the sets  \eqref{holannih1} and \eqref{holannih2} and obtain the set $\mathbb{S}$ of all solutions. Then we restrict $\Phi$ and $\pi_0$ to $\mathbb{S}$ to obtain the description of the solutions of the first BGG operator from Theorem \ref{thm-loc}.

\subsection{Solutions of first BGG operators in local coordinates}\label{section3.4}

Let us show how to use the exponential coordinates and the $\fc$--(co)frame from Section \ref{loc_coord} to describe  the solutions of the first BGG operators in local coordinates.

\begin{thm} \label{thm-coord}
Let $(M,[g])$ be a homogeneous conformal geometry and $\pi: \mathbb{S}\to \mathbb{X}$ and $\Phi: \fk\to \frak{gl}(\mathbb{S})$ be maps describing (local) solutions of the BGG operator on the bundle $\mathcal{X}\to M$ according to Theorem \ref{thm-loc}. If 
$${\sf c}: \fc\to M, \ \ {\sf c}(X):= \exp(L_{1})\dots \exp(L_{j})\exp(X_{\fa})\exp(X_{\fn})o$$ are exponential coordinates compatible with some decomposition $\fk=\fl\oplus \fa\oplus \fn$, then for each $v\in \mathbb{S}$ 
\begin{align} \label{parsec} \tag{Sol}
\pi(\exp(-\Phi(X_{\fn}))\exp(-\Phi(X_{\fa}))\exp(-\Phi(L_{j}))\dots\exp(-\Phi(L_{1}))v)
\end{align}
is a (local) solution of the first BGG operator written in the local trivialization $\fc\times \mathbb{X}$ provided by a $\fc$--(co)frame from Proposition \ref{c-coframe}.
\end{thm}
\begin{proof}
Clearly, $\fc$--(co)frame from Proposition \ref{c-coframe} induces local trivialization $\fc\times \mathbb{X}$ of the bundle $\mathcal{X}.$ The formula for the solution as a function valued in $\mathbb{X}$ follows for our choice of local coordinates  and from Theorem \ref{thm-loc}.
\end{proof}

Let us describe in more detail how the $\fc$--(co)frame provides the (local) trivialization $\fc\times \mathbb{X}$ of the bundle $\mathcal{X}\to M$ with standard fiber $\mathbb{X}$.
\noindent \\[1mm] {$\bullet$}
For $\mathbb{X}$ with trivial action of the grading element $E$, the trivialization provided by a $\fc$--(co)frame has the usual interpretation, i.e., a $\fc$--frame is a trivialization of the tangent bundle,  a $\fc$--coframe is a trivialization of the cotangent bundle and so on for their tensor products. 
\noindent \\[1mm] {$\bullet$}
In this article, we use the trivializations involving spin representations of $\frak{co}(p,q)$ as a description of the spin bundles and do not discuss how these relate with other descriptions of spin bundles.
\noindent \\[1mm] {$\bullet$}
For $\mathbb{X}=\R[w]$, the function $f: \fc\to \R[w]$ corresponds to the section $f\cdot \epsilon^{\frac{-w}{n}}$ of $\mathcal{E}[w]$ in the trivialization provided by the $\fc$--(co)frame $e^*=(e^1,\dots, e^n)$, where $\epsilon:=e^1\wedge \dots \wedge e^n$.
\noindent \\[1mm] {$\bullet$}
In general, the trivialization $\fc\times \mathbb{X}$ can be interpreted as a tensor product of these. 
\\[2mm]
Our exponential coordinates can provide a global covering $\fc\to M$. This leads to an alternative way to determine the set $\mathbb{G}$ of the global solutions. One just needs to check whether the formula \eqref{parsec} assigns the same value to points in $\fc$ covering the same points of $M$.

\subsection{Solutions of first BGG operators for  G\"odel metrics} \label{sol-godel}
Consider the manifolds $M$ with the conformal classes of the G\"odel metrics as described in Sections \ref{ext_godel} and \ref{coord-godel}.

\begin{prop}\label{prop-sol-godel}
\begin{enumerate}
\item There are no normal solutions for any first BGG operator. In particular, there are no Einstein scales nor twistor spinors.
\item The Lie algebra $\fk$ contains all of the conformal Killing fields.
\item There is a $1$--parameter family $$9v_1\epsilon^{\frac{-1}2}=9v_12^{\frac14}e^{\frac{-x}2}(dt\wedge dx \wedge dy \wedge dt)^{\frac{-1}2}$$ of solutions of the first BGG operator on $\mathcal{E}[2]$.
\item There is a $14$--parameter family of conformal Killing $2$--tensors that decomposes into the following $K$--invariant families. 
 \noindent \\ {$\bullet$}
The family of $K$--invariant Killing 2--tensors of the G\"odel metrics
\begin{gather*}
({\scriptstyle \frac14}v_3-3v_2+v_1)\partial_t^2+(2e^{-x}v_2)(\partial_t\partial_y+\partial_y\partial_t)-({\scriptstyle \frac14}v_3-v_1)(\partial_t\partial_z+\partial_z\partial_t)-
\\
v_2\partial_x^2-2v_2e^{-2x}\partial_y^2+({\scriptstyle \frac14}v_3+v_2+v_1)\partial_z^2.
\end{gather*}
{$\bullet$}
The family of Killing $2$--tensors of the G\"odel metrics
\begin{gather*}
-8\sqrt{2}e^{-x}(v_{11}+v_{14})(\partial_t\partial_z+\partial_z\partial_t)+2\sqrt{2}(2yv_{11}+2yv_{14}+v_{10}+v_{13})(\partial_x\partial_z+
\\ \partial_z\partial_x)-
2\sqrt{2}e^{-x}(e^{x}y^2v_{11}+e^{x}y^2v_{14}+ye^{x}v_{10}+ye^{x}v_{13}+e^{x}v_{12}+e^{x}v_9-
\\
2e^{-x}v_{11}-
2e^{-x}v_{14})(\partial_y\partial_z+\partial_z\partial_y).
\end{gather*}
{$\bullet$} The family that does not contain any Killing 2--tensors of the G\"odel metrics
\begin{gather*}
(6yv_7+12v_8y^2+2v_6+3e^{2x}v_4+396e^{-2x}v_8+3v_8e^{2x}y^4+3y^3v_7e^{2x}+3y^2e^{2x}v_6+
\\
3ye^{2x}v_5)\partial_t^2- 48 e^{-x}(4v_8y+v_7)(\partial_t\partial_x+\partial_x\partial_t)
-6e^{-x}(v_8e^{2x}y^4+y^3v_7e^{2x}+y^2e^{2x}v_6+
\\
ye^{2x}v_5+36e^{-2x}v_8-12v_8y^2+e^{2x}v_4-6yv_7-2v_6)(\partial_t\partial_y+\partial_y\partial_t)+(3v_8e^{2x}y^4+
\\
108v_8y^2+12e^{-2x}v_8+3y^3v_7e^{2x}+54yv_7+3y^2e^{2x}v_6+3ye^{2x}v_5+3e^{2x}v_4+18v_6)\partial_x^2-
\\
12e^{-x}(4y^3e^{x}v_8+3e^{x}y^2v_7-8y e^{-x}v_8+2ye^{x}v_6-2 e^{-x}v_7+e^{x}v_5)(\partial_x\partial_y+\partial_y\partial_x)+
\\
6e^{-2x}(5v_8e^{2x}y^4+5y^3v_7e^{2x}+5y^2e^{2x}v_6+5ye^{2x}v_5-12v_8y^2+5e^{2x}v_4+20e^{-2x}v_8-
\\
6yv_7-2v_6)\partial_y^2+(6yv_7+12v_8y^2+2v_6+3e^{2x}v_4+12e^{-2x}v_8+3v_8e^{2x}y^4+
\\
3y^3v_7e^{2x}+3y^2e^{2x}v_6+3ye^{2x}v_5)\partial_z^2.
\end{gather*}
\item There is a $2$--parametric family 
\begin{gather*}
\Big((\cos({\scriptstyle \frac{\sqrt{2}}{2}}z)v_1-\sin({\scriptstyle \frac{\sqrt{2}}{2}}z)v_2)(dtdz-dzdt)+{\scriptstyle \frac{\sqrt{2}}{2}}exp(x)(\cos({\scriptstyle \frac{\sqrt{2}}{2}}z)v_2+
\sin({\scriptstyle \frac{\sqrt{2}}{2}}z)v_1))
\cdot 
\\
(dxdy-
dydx)-e^{x}(sin({\scriptstyle \frac{\sqrt{2}}{2}}z)v_2-cos({\scriptstyle \frac{\sqrt{2}}{2}}z)v_1)(dydz-dzdy)\Big)\cdot 
\\
2^{\frac38}e^{\frac{-3x}4}(dt\wedge dx \wedge dy \wedge dt)^{\frac{-3}4}
\end{gather*} 
of conformal Killing--Yano $2$--forms. None of them is normal nor Killing--Yano $2$--form of the G\"odel metric.
\end{enumerate}
\end{prop}
\begin{proof}
Let us recall that we computed in Lemma \ref{ext-godel} the normal conformal extension $\alpha:\fk \to \sodc$ corresponding to the conformal class of G\"odel metrics as $\alpha(x_1,x_2,x_3,x_4,x_5)=$
$$
 \left[ \begin {smallmatrix} 0&-\frac{1}{2}{ x_1}+\frac{1}{6}{ x_4}&-\frac{1}{12}{
 x_2}&-\frac{1}{12}{ x_3}&\frac{1}{6}{ x_1}-\frac{1}{8}{ x_4}&0
\\ 
{ x_1}&0&\frac{\sqrt {2}}{4}{ x_3}&-\frac{\sqrt {2}}{4}
{ x_2}&0&-\frac{1}{6}{ x_1}+\frac{1}{8}{ x_4}\\ { x_2}&\frac{\sqrt {2}}{2}{ x_3}&0&\frac{\sqrt {2}}{2}{ x_1}-{ x_3}-\frac{\sqrt {2}}{4}
{ x_4}+\sqrt {2}{ x_5}&-\frac{\sqrt {2}}{4}{ x_3}&\frac{1}{12}{ x_2}
\\ 
{ x_3}&-\frac{\sqrt {2}}{2}{ x_2}&-\frac{\sqrt {2}}{2}{
 x_1}+{ x_3}+\frac{\sqrt {2}}{4}{ x_4}-\sqrt {2}{ x_5}&0&\frac{\sqrt {2}}{4}
{ x_2}&\frac{1}{12}{ x_3}\\ { x_4}&0&-\frac{1}{2}
\sqrt {2}{ x_3}&\frac{\sqrt {2}}{2}{ x_2}&0&\frac{1}{2}{ x_1}-\frac{1}{6}{ x_4}
\\ 
0&-{ x_4}&-{ x_2}&-{ x_3}&-{ x_1}&0
\end {smallmatrix} \right] 
$$
with the curvature $\kappa(\alpha(x_1,x_2,x_3,x_4,x_5),\alpha(y_1,y_2,y_3,y_4,y_5))=$
$$
 \left[ \begin {smallmatrix} 
0&-\frac{\sqrt {2}}{2} z_{23} &
- \frac{\sqrt {2}}{4}z_{13}
-\frac{\sqrt {2}}{8}z_{34}
&\frac{\sqrt{2}}{4}z_{12}+\frac{\sqrt{2}}{8}z_{24} &\frac{\sqrt {2}}{4} z_{23} &0
\\ 
0&\frac13 z_{14}&-\frac{1}{6}z_{12}&
-\frac{1}{6}z_{13}&0&-\frac{\sqrt {2}}{4}
 z_{23}
\\
0&-\frac{1}{6}z_{24}&0&\frac13z_{23}&\frac{1}{6}z_{12}&\frac{\sqrt {2}}{8}z_{34}+\frac{\sqrt {2}}{4}
z_{13}
\\ 
0&-\frac{1}{6}z_{34}&\frac13z_{23}&0&\frac{1}{6}z_{13}&-\frac{\sqrt {2}}{8}z_{24}-\frac{\sqrt {2}}{4}z_{12}
\\ 
0&0&\frac{1}{6}z_{24}&\frac{1}{6}z_{34}&-\frac13z_{14}&\frac{\sqrt {2}}{2} z_{23}
\\ 
0&0&0&0&0&0
\end {smallmatrix} \right],
$$
where we write $z_{ij}=x_iy_j-x_jy_i.$ 

The formulas for $\alpha$ and $\kappa$ allow us to compute directly the infinitesimal holonomy
$$
hol(\alpha)=\sodc
$$
and therefore, there are no normal solutions and claim (1) follows.

To show the claim (2), let us start with the conformal Killing vectors, i.e., $\mathbb{V}=\sodc$, $\mathbb{X}=\R^4$. It is well known that the tensor $\Psi$ determining the prolongation connection is the insertion $X\mapsto -\kappa(\alpha(X),.)$ into curvature, \cite{deform}. Thus, $\Phi(X)(\alpha(Y)+W)=\{\alpha(X),W\}+\alpha([X,Y])$ for $X,Y\in \fk,W\in \fp$ and the curvature of the prolongation connection simplifies as \begin{gather*}
R^{\Phi}(\alpha(X),\alpha(Y))(\alpha(Z)+W)=
\{\kappa(\alpha(X),\alpha(Y)),W\}-
\\
\kappa(\alpha(X),\{\alpha(Y),W\})+\kappa(\alpha(Y),\{\alpha(X),W\})
\end{gather*}
using the Jacobi identity. In particular, this ensures that all elements of the image of $\alpha$ correspond to conformal Killing fields and it is easy to compute that there is no nonzero $W$ annihilated by $R^{\Phi}$ and the claim (2) follows.

We compute the remaining claims (3), (4) and (5) in the following steps. Firstly,  we realize in Maple the algorithm from Section \ref{section3.3} and describe the solutions of the first BGG operators according to Theorem \ref{thm-loc}. Next, in order to present the result obtained according to Theorem \ref{thm-coord} in the exponential coordinates $\fc\to M$ from Lemma \ref{lem-coord-godel} in the original coordinates, we used the transition $t=a_2,  x=2a_3, y=a_4e^{-2a_3}, z=a_1$ between these two coordinates from Lemma \ref{lem-cartan-godel}. This translates the functions $\fc\to \mathbb{X}$ to functions $M\to \mathbb{X}$ and provides the solutions using the following $\fc$--(co)frame we computed in Lemma \ref{lem-cartan-godel}
\begin{gather*}
e_1:=\partial_t+\partial_z,\ \ e_2:=\partial_x, \ \ e_3:=\sqrt{2}(e^{-x}\partial_y-\partial_t), \ \ e_4:=-\frac12\partial_t+\frac12\partial_z, \\
e^1:=\frac12 (dt +e^x dy+dz), \ \ e^2:=dx, \ \ e^3:=\frac{e^x}{\sqrt{2}} dy, \ \  e^4:=-dt-e^xdy+dz.
\end{gather*}
Let us emphasize that for the global existence of the solutions in the case $M=S^1\times \R^3$, we checked whether the solutions are periodic in $t$ with the period $4\sqrt{2}\pi$ and indeed, this is the case for the solutions we computed.

For the claim (3), we consider $\mathbb{V}=\bigodot^2_0\mathbb{T}$ which does not correspond to a well--known BGG operator. Therefore in Appendices \ref{bgg-const} and \ref{apendix}, we provide some more details about this BGG operator and the prolongation connection. We do not give such details for the remaining cases, because it would be even more complicated.

Altogether, we compute that $\mathbb{S}$ is trivial representation $\Phi$ of $\fk$ and consist of the following elements of $\mathbb{V}$
\[
\left[ \begin {smallmatrix} {\frac {7}{16}}{ v_1}&0&0&0&0&{ \frac{1}{4}}
{ v_1}\\ 0&{\frac {9}{8}}{ v_1}&0&0&-{ \frac{5}{4}}{
 v_1}&0\\ 0&0&{ v_1}&0&0&0\\ 0
&0&0&{ v_1}&0&0\\ 0&-{ \frac{5}{4}}{ v_1}&0&0&{ \frac{9}{2}}{ 
v_1}&0\\ { \frac{1}{4}}{ v_1}&0&0&0&0&9{ v_1}
\end {smallmatrix} \right]\subset \left[ \begin{smallmatrix}
\mathbb{V}_{2}& *& *\\ \mathbb{V}_1 & \mathbb{V}_0 & *\\ *& \mathbb{V}_{-1}& \mathbb{X}=\mathbb{V}_{-2}\end{smallmatrix} \right].
\]

Therefore, the solutions are constant functions $\fc\to \R[2]$, which have the claimed form as a section of $\mathcal{E}[2]$.

For the claim (4), we consider $\mathbb{V}=\boxtimes(\bigodot^2\sodc)$ and fix the parametrization $(s_1,\dots,s_9)$ of the projective slot given by the $\fc$--frame as
\begin{gather*}
\bigodot{}^2_0TM=\{ s_1 e_1^2+s_2e_1e_2+s_2e_2e_1+s_3e_1e_3+s_3e_3e_1+s_4e_1e_4+s_4e_4e_1\\
+s_5e_2^2+s_6e_2e_3+s_6e_3e_2+s_7e_2e_4+s_7e_4e_2\\
- (2s_4+s_5)e_3^2+s_8e_3e_4+s_8e_4e_3+s_9e_4^2 \}.
\end{gather*}
We compute that the radical acts trivially in the representation $\Phi$ on the solutions and they exist globally in both cases $M=\R^4$ and $M=S^1\times \R^3$. Therefore, we can express the results of our computations as representations $\Phi$ of $\frak{sl}(2,\R)$, which we analyze using the standard procedures from the representation theory. Altogether, $dim(\mathbb{S})=14$ and
\begin{enumerate}
\item there is a $3$--dimensional trivial $\frak{sl}(2,\R)$--representation $\R^3=\langle v_1,v_2,v_3 \rangle$ with the projection 
$$(v_1,0,0,v_2,-v_2,0,0,0,v_3)$$
that correspond to the first family of solutions,
\item there is a $5$--dimensional $\frak{sl}(2,\R)$--representation $\bigodot^4 \R^{2*}$ parametrized as $v_4 y_1^4+v_5 y_1^3y_2+v_6y_1^2y_2^2+v_7y_1y_2^3+v_8y_2^4$ with the projection
\begin{gather*}
(24v_8+4v_6+6v_4,
 -12v_7-6v_5,
 -24\sqrt{2}v_8+6\sqrt{2}v_4,
 -36v_8-6v_6-9v_4, 
12v_8+
\\
18v_6+3v_4, 
12\sqrt{2}v_7-6\sqrt{2}v_5, 
24v_7+12v_5, 
48\sqrt{2}v_8-12\sqrt{2}v_4, 
96v_8+16v_6+24v_4)
\end{gather*}
that corresponds to the last family of solutions, and
\item  there is a $6$--dimensional $\frak{sl}(2,\R)$--representation consisting of two copies of $\bigodot^2 \R^{2*}$ that we parametrize as $v_9 y_1^2+v_{10} y_1y_2+v_{11}y_2^2$ and $v_{12} y_1^2+v_{13} y_1y_2+v_{14}y_2^2$ with the projection
\begin{gather*}
\big(-\sqrt{2}(2v_{11}+v_{12}+2v_{14}+v_9),
 \sqrt{2}(v_{10}+v_{13}), 
2v_{14}-v_{12}+2v_{11}-v_9, 
0, 
0, 
0, 
\\
2\sqrt{2}(v_{10}+v_{13}), 
4v_{14}-2v_{12}+4v_{11}-2v_9, 
4\sqrt{2}(2v_{11}+v_{12}+2v_{14}+v_9)\big)
\end{gather*}
that corresponds to the second family of solutions.
\end{enumerate}

For the claim (5), the situation with our computations is analogous to the proof of claim (4). We fix the parametrization $(s_1,\dots,s_6)$ of the projective slot using the $\fc$--coframe as
\begin{gather*}
\wedge^2T^*M[3]=\{ \epsilon^{\frac{-3}4}(s_1(e^1e^2-e^2e^1)+ s_2(e^1e^3-e^3e^1)+ s_3(e^1e^4-e^4e^1)+ 
\\
s_4(e^2e^3-e^3e^2)+ s_5(e^2e^4-e^4e^2)+ s_6(e^3e^4-e^4e^3)) \}.
\end{gather*}
The coordinate $a_1$ in the radical is the only part with nontrivial action $\Phi$ on $\mathbb{S}$ and thus solutions 
exist globally in both cases $M=\R^4$ and $M=S^1\times \R^3$. Altogether, $dim(\mathbb{S})=2$ and there is a $2$--dimensional representation with projection 
$$(0,0,v_1,v_2,0,0),$$
where 
\begin{gather*}
\exp\big(-\Phi(a_1,a_2,a_3,a_4,a_5)\big)(0,0,v_1,v_2,0,0)=
\\
(0,0,\cos(\frac{\sqrt{2}}{2}a_1)v_1-
\sin(\frac{\sqrt{2}}{2}a_1)v_2 , \sin(\frac{\sqrt{2}}{2}a_1)v_1+\cos(\frac{\sqrt{2}}{2}a_1)v_2,0,0)
\end{gather*}
and the claim follows.
\end{proof}

To provide some more insight into the algorithm from Section \ref{section3.3}, let us describe how the claim `There are no Einstein scales nor twistor spinors on $M$ with the conformal class of the G\"odel metrics' can be proved using steps (1),(2),(5),(6) of the algorithm. Of course, the prolongation connection coincides with the tractor connection in these two cases and thus step (5) is trivial and $\Psi=0$.

Let us first consider the standard representation $\mathbb{T}=\R^{6}$, where $\rho$ is just multiplication by the given matrix in $\sodc$. Thus the eigenspaces of the action $\rho(E)$ have $(1,4,1)$--block structure $$
\left[ \begin{smallmatrix}\mathbb{T}_1\\ \mathbb{T}_0\\\mathbb{T}_{-1}\end{smallmatrix} \right]$$ with $\mathbb{X}=\mathbb{T}_{-1}$ being the projective slot. Then we immediately see that the curvature $\rho \circ \kappa$ annihilates only $\mathbb{T}_1$ and thus $\mathbb{S}^0=\mathbb{T}_1$. Since the image of $\alpha$ does not preserve $\mathbb{T}_1$, we conclude that $\mathbb{S}^1=0$ and there are no normal solutions. 

Further, let us consider the spin representation $\sodc\to \frak{su}(2,2)\subset \frak{gl}(4,\C)$ (acting on $\mathbb{D}=\C^4$) as
\begin{gather*}
 \left[ \begin {smallmatrix} a&b&c_1&c_2&d&0\\
u&a_{11}&e_{11}&e_{12}&0&-d
\\ v_1&c_{11}&0&b_{12}&-e_{11}&-c_1
\\ v_2&c_{21}&-b_{12}&0&-e_{12}&-c_2
\\ w&0&-c_{11}&-c_{21}&-a_{11}&-b
\\ 0&-w&-v_1&-v_2&-u&-a
\end{smallmatrix} \right]\mapsto 
\\
 \left[ \begin {smallmatrix} 
{ \frac12}(a+a_{11}- ib_{12})&-2\sqrt {2}(e_{11}+ie_{12})&-{ \frac{\sqrt {2}}{2}}(c_1+i
c_2)&4d\\ 
- { \frac{\sqrt {2}}{8}} (c_{11}-ic_{21})&\frac{1}{2}(a-a_{11}+ib_{12})&- { \frac14} b&- { \frac{\sqrt {2}}{2}} (c_1-ic_2)\\ 
-\frac{\sqrt {2}}{2}(v_1-i v_2)&-4
u&-({ \frac{1}{2}}a- a_{11}-ib_{12})&2\sqrt {
2}(e_{11}-ie_{12})\\ 
{ \frac14} w&- { \frac{\sqrt {2}}{2}} (v_1+iv_2)& { \frac{\sqrt {2}}{8}}( c_{11}
+i c_{21})&- { \frac12} (a+ a_{11}+ib_{12})\\ 
\end {smallmatrix} \right],
\end{gather*}
where $\frak{su}(2,2)$ corresponds to the pseudo--Hermitian form $$((z_0,z_1,z_2,z_3),(w_0,w_1,w_2,w_3))=z_0\bar{w}_3+z_3\bar{w}_0+z_1\bar{w}_2+z_2\bar{w}_1.$$ Thus $\rho\circ \alpha(x_1,x_2,x_3,x_4,x_5)$ equals to
\[
 \left[ \begin {smallmatrix} 
-iz& -x_3+i x_2 & { \frac{\sqrt{2}}{24}}(x_2+ix_3)& { \frac{2}{3}}x_1-{ \frac{1}{2}}x_4\\
- { \frac{1}{8}}(x_3+ix_2)&iz& { \frac{1}{8}}x_1-{ \frac{1}{24}}x_4& { \frac{\sqrt{2}}{24}}(x_2-ix_3)\\
  -{ \frac{\sqrt{2}}{2}}(x_2-ix_3)& -4x_1&i z & x_3+i x_2\\
  { \frac{1}{4}}x_4&  -{ \frac{\sqrt{2}}{2}}(x_2+i x_3)& { \frac{1}{8}}(x_3-i x_2)&-iz
\end {smallmatrix} \right],
\]
 where $z= { \frac{\sqrt{2}}{4}}x_1-{ \frac{1}{2}}x_3-{ \frac{\sqrt{2}}{8}}x_4+{ \frac{\sqrt{2}}{2}}x_5$, and writing $z_{ij}=x_iy_j-x_jy_i$, $\rho\circ \kappa$ equals to
\[
\left[ \begin {smallmatrix} 
\frac{1}{6}(z_{14}-iz_{23})&
\frac{\sqrt{2}}{3}(z_{12}+iz_{13}) &
\frac{1}{4}(-z_{13}+iz_{12})+\frac{1}{8}(-z_{34}+iz_{24})&
-\sqrt {2} z_{23} 
\\ 
\frac{\sqrt {2}}{48} (z_{24}-iz_{34}) &
\frac{1}{6}(-z_{14}+iz_{23})&
-\frac{\sqrt {2}}{8}
z_{23} &
-\frac{1}{4}(z_{13}+iz_{12})-\frac{1}{8}(z_{34}+iz_{24})
\\ 
0&0&\frac{1}{6}(z_{14}+iz_{23})&
\frac{\sqrt {2}}{3}  (-z_{12}+iz_{13}) 
\\ 
0&0&
-\frac{ \sqrt {2}}{48} (z_{24}+iz_{34}) &
-\frac{1}{6}z_{14}
\end {smallmatrix}
 \right] 
.
\]
We again see that there are no normal solutions.

\section{Applications of normal solutions and holonomy reductions} \label{section4}

\subsection{Holonomy reductions}

Let $G$ be a Lie group with the Lie algebra $\so$ such that the representation $\rho: \so\to \frak{gl}(\mathbb{V})$ integrates to a representation $\lambda:G\to Gl(\mathbb{V}).$ Consider a homogeneous conformal geometry $(K/H,[g])$ with an associated conformal extension $\alpha$ of $(\fk,\fh)$ such that $\alpha$ restricted to $\fh$ integrates to a Lie group homomorphism $\iota: H\to P$. Then connections on the tractor bundle $\mathcal{V}:=K\times_{\lambda \circ\iota(H)} \mathbb{V}$ are in one--to--one correspondence with $Gl(\mathbb{V})$--principal connections on the bundle $K\times_{\lambda \circ\iota(H)} Gl(\mathbb{V}).$ The tractor connection $\nabla^{\rho\circ \alpha}$ provides a reduction $K\times_{\lambda \circ\iota(H)} \lambda(G) \subset K\times_{\lambda \circ\iota(H)} Gl(\mathbb{V})$. For this reduction, the tractor bundle admits a non--linear decomposition into $G$--orbits $K\times_{\lambda \circ\iota(H)} \mathcal{O}_{[v]}\subset \mathcal{V}$ of type $\mathcal{O}_{[v]}=G/G_v=\lambda(G)v\subset \mathbb{V}$, where $G_v$ is the stabilizer of $v$.

\begin{def*}
We say that an $H$--equivariant function $s: K\to \mathcal{O}_{[v]}\subset \mathbb{V}$ parallel with respect to the induced (non--linear) connection on $K\times_{\rho\circ \iota(H)}\mathcal{O}_{[v]}$ is a \emph{holonomy reduction of $G$--type} $\mathcal{O}_{[v]}$.
\end{def*}

If we consider the Cartan geometry $(\ba=K\times_{\iota(H)}P,\omega_\alpha)$ of type $(G,P)$ on $M=K/H$ as described in Section \ref{cartan} by the maps $\alpha$, $\iota$ and element $u\in \ba$, then the holonomy reduction of $G$--type $\mathcal{O}_{[v]}$ can be naturally extended to $P$--equivariant function $s: \ba\to \mathcal{O}_{[v]}$. Then for any coset $\beta$ in $P\backslash G/G_v$ with the representative $w\in \mathbb{V}$, there is  \begin{itemize}
\item an initial submanifold $M_\beta$ of $M$ consisting of $kH$ such that $\lambda(p)^{-1}(s(k))=w$ for some $p\in P$, and 
\item a $P_w:=P\cap G_w$--bundle $\ba_w$ over $M_\beta$ consisting of all $kup\in \ba$ such that $\lambda(p)^{-1}(s(k))=w.$ 
\end{itemize}
This defines a \emph{curved orbit decomposition} $M=\bigcup_{\beta\in P\backslash G/G_v}M_\beta$ to $P$--types $\beta$. The basic results on the holonomy reductions \cite[Section 2.3]{hol} and \cite[Theorem 2.6]{hol} can be reformulated  in the homogeneous setting as follows.

\begin{prop}\label{hol-red}
Let $s: K\to \mathbb{V}$ be an $H$--equivariant section corresponding to a normal solution $v\in \mathbb{N}$. Then $s$ restricts to a holonomy reduction of type $\mathcal{O}_{[v]}$. Moreover, for any representative $w\in \mathbb{V}$ of $P$--type $\beta$, there is Cartan geometry $(\ba_{w},\omega_w)$ of type $(G_w,P_w)$ uniquely determined by the property $\omega_w:=\omega_\alpha\circ Tj$ for the natural inclusion $j: \ba_{w}\to K\times_{\iota(H)} P$. For $kup\in \ba_w$, $\Ad_p^{-1}(\alpha(\fk))\cap \fg_w$ is the Lie algebra of the Lie subgroup of $K$ preserving the Cartan geometry $(\ba_{w},\omega_w).$

Conversely,  the inclusion $K\times_{\lambda \circ\iota(H)} \mathcal{O}_{[v]}\subset \mathcal{V}$ induced by a holonomy reduction $s$ of type $\mathcal{O}_{[v]}$ provides a section of the tractor bundle parallel for the tractor connection and thus a normal solution $v=s(e)\in \mathbb{N}$ of the corresponding first BGG operator.
\end{prop}
\begin{proof}
We conclude from \cite[Section 2.3]{hol} that normal solutions of first BGG operators are in one--to--one correspondence with holonomy reductions and the description $\ba=K\times_{\iota(H)}P$ then implies using \cite[Theorem 2.6]{hol} the claimed construction of Cartan geometries $(\ba_w\to M_\beta,\omega_w)$ of type $(G_w,P_w)$. Thus the remaining claims follow from homogeneity and construction of the bundles $\ba_w.$
\end{proof}

Let us summarize how we find and interpret the holonomy reductions in practice.
\noindent \\[1mm] {\bf (1)}
 We start with the normal solution $v\in \mathbb{N}\subset \mathbb{V}$ and extend it to the holonomy reduction $s: \fc \to \mathcal{O}_{[v]}$ using the exponential coordinates ${\sf c}:\fc \to M$.
\noindent \\[1mm] {\bf (2)}
We determine  $P$--types of points $X$ of $\fc$ 
and for the fixed representative $w\in \mathbb{V}$ of $P$--type, we find a representative $p_{X,w}\in P$ such that $\lambda(p_{X,w})^{-1}(s(X))=w.$
\noindent \\[1mm] {\bf (3)}
 Let $G_0$ be the maximal subgroup of $P$ with the Lie algebra $\frak{co}(p,q)$ and let $G_{+,w}$ and $G_{0,w}$ be the kernel and the image of $P_w$ for the projection $P_w\to P/\exp(\R^{n*})\cong G_0$. The coframe of $TM$  obtained by the adjoint action of $p_{X,w}^{-1}$ on the $\fc$--coframe provides an underlying $G_{0,w}$--structure on the curved orbit.
\noindent \\[1mm] {\bf (4)}
 Moreover, the set of smooth functions $\fc\to G_{+,w}\subset \exp(\R^{n*})$ corresponds to a distinguished set of Weyl connections that are connections on the $G_{0,w}$--structure. To compute this set explicitly, one starts with the Weyl connection provided by the description \eqref{Cartan_connection} of the corresponding Cartan connection from Proposition \ref{localext} and interprets the smooth function $\fc \to \R^{n*}$ as the change of the Weyl connection in the usual way, \cite[Section 1.6]{parabook}. 
\noindent \\[1mm] {\bf (5)}
 Let us emphasize that the normality of the original conformal geometry has a consequence that from the viewpoint of $G_{0,w}$--structures, the connections have special curvature.

In particular, there are the following interesting cases \cite{HS-twistor, CG-Fefferman, hol}.
\noindent \\[1mm]
{$\bullet$} Einstein scales with $v\in \mathbb{N}$ such that $\bg(v,v)\neq 0$ provide decompositions to points of three $P$--types determined by positivity, negativity or vanishing of the scalar product $\bg(v,s)$ for the value of the corresponding function $s: K\to \mathcal{O}_{[v]}$. If we denote $M^+\cup M^-\cup M^0$ the corresponding curved orbits, then on the open orbits $M^+\cup M^-$, there is the Einstein metric having Einstein constant with opposite sign than $\bg(v,v)$. The closed orbit $M^0$ is a hypersurface separating $M^+$ and $M^-$ carrying a conformal structure.
\noindent \\[2mm] {$\bullet$}
  Einstein scales with $v\in \mathbb{N}$ such that $\bg(v,v)= 0$ provide decompositions to points of five  $P$--types such that the value of $s$ is positive or negative multiple of $v$, or the value of $s$ belongs to the orthocomplement of $v$ or $\bg(v,s)$ is positive or negative for the value of $s$. If we denote $M^{0,+}\cup M^{0,-}\cup M^{0,\perp} \cup M^+\cup M^-$ the corresponding curved orbits, then on the open orbits $M^+\cup M^-$, there is a Ricci flat metric. The closed orbits $M^{0,+}\cup M^{0,-}\cup M^{0,\perp}$ consist (if $p\neq 0$) of smooth embedded hypersurface $M^{0,\perp}$ with (point) edges $M^{0,+}\cup M^{0,-}$ or (if $p=0$)  isolated points.
\noindent \\[2mm] {$\bullet$}
 In signatures $(2,3)$ and $(3,3)$, generic twistor spinors provide a curved orbit decomposition such that the open orbits carry a generic rank two, or three null--distribution on $5$-- or $6$--manifold, respectively.
\noindent \\[2mm] {$\bullet$}
 Normal conformal Killing fields $v\in \mathbb{N}$ such that $v$ is a non--degenerate tractor 2--form (i.e., $n$ is even)  provide a curved orbit decomposition such that the open orbits are locally Fefferman spaces of almost CR manifolds.
\noindent \\[2mm]
Since the conformal class of the G\"odel metrics does not admit any normal solutions, we need to consider different conformal geometries to provide non--trivial examples of holonomy reductions.

\subsection{Holonomy reductions for submaximal pp--wave} \label{pp-hol}
Let us consider the conformal class of the submaximally symmetric pp--wave of signature $(1,3)$, \cite{submax,kundt},
\[
g=x^2dt^2+2dtdz+dx^2+dy^2
\]
on $M=(t,x,y,z)=\R^4$.
Among the seven conformal Killing fields, we pick the following $4$--dimensional solvable Lie algebra $\fk$ generated by
\[
k_1:=\partial_t,\ \ k_2:=e^{-t}(\partial_x+x\partial_z),\ \ k_3:=\partial_y,\ \ k_4:=\partial_z.
\]
Since this is an orthonormal frame of $g$ at the origin $o=(0,0,0,0)$, we can directly compute the associated conformal extension.

\begin{lem}
Suppose $(x_1,x_2,x_3,x_4)$ is the parametrization of $\fk$ via the frame $k_1,k_2,k_3,k_4$. Then $$\alpha(x_1,x_2,x_3,x_4)= \left[ \begin {smallmatrix} 0&{ \frac{1}{2}}{ x_1}&0&0&0&0
\\ { x_1}&0&0&0&0&0\\ { x_2}&-
{ x_2}&0&0&0&0\\ { x_3}&0&0&0&0&0
\\ { x_4}&0&{ x_2}&0&0&-{ \frac{1}{2}}{ x_1}
\\ 0&-{ x_4}&-{ x_2}&-{ x_3}&-{ x_1}&0
\end {smallmatrix} \right]
$$
is the normal conformal extension associated with the conformal class of the pp--wave $g$ with curvature $\kappa(\alpha(x_1,x_2,x_3,x_4),\alpha(y_1,y_2,y_3,y_4))$ of the form
$$ \left[ \begin {smallmatrix} 0&0&0&0&0&0\\ 0&0&0&0
&0&0\\ 0&{ \frac{1}{2}}({ x_1}{ y_2}-{ x_2}{
 y_1})&0&0&0&0\\ 0&{ \frac{1}{2}}({ x_3}{ y_1}-{
 x_1}{ y_3})&0&0&0&0\\ 0&0&{ \frac{1}{2}}({ x_2}{ 
y_1}-{ x_1}{ y_2})&{ \frac{1}{2}}({ x_1}{ y_3}-{ x_3}{
 y_1})&0&0\\ 0&0&0&0&0&0\end {smallmatrix} \right]
$$
\end{lem}
\begin{proof}
Since $\fk=\fc$, the construction from Proposition \ref{conf_ext} simplifies to finding the maps $\alpha_0,\alpha_1$, that are uniquely determined in the given form by the normalization conditions \eqref{norm} and vanishing of the $a$--part in \eqref{blockso} of $\alpha_0$.
\end{proof}

Let us show that this conformal geometry admits normal solutions that we can use for finding  holonomy reductions.

\begin{prop} \label{pp-sol}
Suppose $\rho=w_1\lambda_1+w_2\lambda_2+w_3\lambda_3$  for the  fundamental weights $\lambda_i$ of the complexification of $\sodc$. 
\begin{itemize}
\item If $w_2=w_3$, then $\Phi|_{\mathbb{N}}$ is obtained by branching the representation $w_1\lambda_1$ of $\frak{sl}(2,\R)$ to $\left[ \begin {smallmatrix} 0& -\frac12 x_1\\ -x_1& 0 \end {smallmatrix} \right].$ 
\item If $w_2\neq w_3$, then $\Phi|_{\mathbb{N}}$ is obtained by branching the complexification of the representation $w_1\lambda_1$ of $\frak{sl}(2,\R)$ to $\left[ \begin {smallmatrix} 0& -\frac12 x_1\\- x_1& 0 \end {smallmatrix} \right].$ 
\item In both cases, the projection $\pi: \mathbb{N}\to \mathbb{X}$ is induced by the identification of $\mathbb{N}$ with $\frak{sl}(2,\R)\oplus \frak{so}(2)$--orbit of the lowest weight vector, where $\frak{sl}(2,\R)\oplus \frak{so}(2)\subset \sodc$ corresponds to the diagonal in the decomposition of $\sodc$ into $2\times 2$--blocks.
\end{itemize}
In particular, there is 
\begin{enumerate}
\item[(Es) ]a $2$--parameter family of Einstein scales with $$\Phi(x_1,x_2,x_3,x_4)=\left[ \begin {smallmatrix} 0& -\frac12 x_1\\- x_1& 0 \end {smallmatrix} \right]\subset \frak{gl}(2,\R)$$ and $\pi((v_1,v_2)^t)=v_2,$ 
\item[(ts)] a $2$--parameter family of twistor spinors corresponding to constant function $(0,v_1+iv_2)^t$, and
\item[(cKf)] a $1$--parameter family of  normal conformal Killing fields corresponding to constant functions valued in $\sodc$ decomposed as \eqref{blockso}  with $w=v_1$ and 
remaining elements vanish.
\end{enumerate}
\end{prop}
\begin{proof}
It is a simple observation that the image of $\kappa$ is $2$--dimensional and the bracket of the image of $\alpha$ with it is $3$--dimensional subspace of $\R^4$. Doing further bracketing, we get nothing new and 
\[
hol(\alpha)= \left[ \begin {smallmatrix} 0&0&0&0&0&0\\ 0&0&0&0
&0&0\\ { h_3}&{ h_1}&0&0&0&0
\\ { h_4}&{ h_2}&0&0&0&0\\ {
 h_5}&0&-{ h_1}&-{ h_2}&0&0\\ 0&-{ h_5}&-{
 h_3}&-{ h_4}&0&0\end {smallmatrix} \right].
\]
Thus $hol(\alpha)$ is a Heisenberg Lie algebra corresponding to the negative part of the contact grading of $\sodc$. If we denote by $E_2$ the corresponding grading element, the normal solutions belong to the eigenspace of $E_2$ of the lowest weight vector that is an irreducible $\frak{sl}(2,\R)\oplus \frak{so}(2)$--module. For $w_2=w_3$, it is $\frak{sl}(2,\R)$--module with the highest weight $w_1\lambda_1$ and trivial $\frak{so}(2)$--module. For $w_2\neq w_3,$ it is the complexification of $\frak{sl}(2,\R)$--module with the highest weight $w_1\lambda_1$ and $\frak{so}(2)$ acts as the imaginary part of $\mathbb{C}$. The branching is then just the restriction to the last block in the diagonal in the decomposition of the image of $\alpha$ into $2\times 2$--blocks. For the particular weights $\lambda_1$ of the standard representation, $\lambda_2$ of the spinor representation and $\lambda_2+\lambda_3$ of the adjoint representation, we obtain the claimed normal solutions.
\end{proof}

As in the case of the G\"odel metric, we present all the results in the original coordinates rather than the exponential coordinates ${\sf c}: \fc\to M$. So let us compute the $\fc$--coframe in the original coordinates.

\begin{lem}
There is the following $\fc$--(co)frame on $M$
\begin{gather*}
e^1=dt,\ \  e^2=xdt+dx,\ \ e^3=dy,\ \ e ^4=-xdx+dz,\\
 e_1=\partial_t-x\partial_x-x^2\partial_z, \ \ e_2=\partial_x+x\partial_z,\ \ e_3=\partial_y,\ \ e_4=\partial_z.
\end{gather*}
In particular, the Einstein scales take form  \begin{gather*}
(\sinh({ \frac{\sqrt{2}}{2}}t)\sqrt{2}v_1+\cosh({ \frac{\sqrt{2}}{2}}t)v_2) e^{\frac{t}4} (dt \wedge dx \wedge dy \wedge dz)^{\frac{-1}{4}}
\end{gather*}
 and the normal conformal Killing vectors take form $v_1\partial_{z}$.
\end{lem}
\begin{proof}
Since the composition of exponential maps corresponds to composition of flows of the conformal Killing fields, we compute  ${\sf c}: \fc\to M$, $(a_1,a_2,a_3,a_4)\mapsto (a_1,a_2,a_3,\frac12 a_2^2+a_4)$. If we consider the matrix representation 
$$ \left[ \begin {smallmatrix} { \frac{1}{2}}{ a_1}&{ a_2}&0&0
\\ 0&-{ \frac{1}{2}}{ a_1}&0&0\\ 0&0&{
 a_3}&0\\ 0&0&0&{ a_4}\end {smallmatrix} \right] 
$$
of $\fk$, then the Maurer--Cartan form in the exponential coordinates takes form
$$ \left[ \begin {smallmatrix} { \frac{1}{2}}{ da_1}&{ a_2da_1+da_2}&0&0
\\ 0&-{ \frac{1}{2}}{ da_1}&0&0\\ 0&0&{
 da_3}&0\\ 0&0&0&{ da_4}\end {smallmatrix} \right].
$$
Then we can push-pull the corresponding $\fc$--(co)frame to $M$ and obtain the claim of the lemma.

Using the explicit formula for the coordinates, we can translate the results computed according to Theorem \ref{thm-coord} in the exponential coordinates to the original coordinates and obtain the claimed form for the normal solutions.
\end{proof}

Let us discuss the holonomy reductions induced by Einstein scales, twistor spinors, and normal conformal Killing fields from Proposition \ref{pp-sol}. Let us start with the Einstein scales.

\begin{prop}
All the Einstein scales have the $G$--type corresponding to a null--vector. 
\begin{itemize}
\item If $\frac{\sqrt{2}v_1-v_2}{\sqrt{2}v_1+v_2}>0$, then there are curved orbits $M^{0,\perp}\cup M^+\cup M^-$ determined by 
$t=\frac1{\sqrt{2}}ln(\frac{\sqrt{2}v_1-v_2}{\sqrt{2}v_1+v_2})$, $t>\frac1{\sqrt{2}}ln(\frac{\sqrt{2}v_1-v_2}{\sqrt{2}v_1+v_2}) $ and $t< \frac1{\sqrt{2}}ln(\frac{\sqrt{2}v_1-v_2}{\sqrt{2}v_1+v_2})$, respectively. 
\item If $\frac{\sqrt{2}v_1-v_2}{\sqrt{2}v_1+v_2}\leq 0$ then all of the points of $M$ have the $P$--type corresponding to open orbit. 
\end{itemize}
The open orbits carry a Ricci flat metric \begin{gather*}
\Big({ \frac{\sqrt{2}}{2}}(e^{-{ \frac{\sqrt{2}}{2}}t}-e^{{ \frac{\sqrt{2}}{2}}t})v_1+{ \frac{1}{2}}(e^{-{ \frac{\sqrt{2}}{2}}t}+e^{{ \frac{\sqrt{2}}{2}}t})v_2\Big)^{-2}g
\end{gather*} 
in the conformal class and the closed orbit carries a Cartan geometry of type $(SO(1,3)\rtimes \R^4,P_1\rtimes \R^3),$ where $\R^4$ decomposes as the standard tractor bundle for $G=SO(1,3)$  into a null--line preserved by $P_1$ and its orthocomplement $\R^3$. The underlying geometric structure on the closed orbit consists of
\begin{enumerate}
\item a $1$--dimensional distribution $v_1\partial_z$ with projection $$(\frac1{\sqrt{2}}ln(\frac{\sqrt{2}v_1-v_2}{\sqrt{2}v_1+v_2}),x,y,z)\mapsto (x,y)$$ on the leaf space, and
\item the flat conformal class $[dx^2+dy^2]$ on the leaf space.
\end{enumerate}
\end{prop}
\begin{proof}
We can deduce from the data in Proposition \ref{pp-sol} that the holonomy reduction corresponding to the Einstein scale $(v_1,v_2)^t\in \mathbb{N}$ is 
\begin{gather*}(t,x,y,z)\mapsto \big(0,0,0,0,
\\
\cosh({ \frac{\sqrt{2}}{2}}t)v_1+{ \frac{\sqrt{2}}{2}}\sinh({ \frac{\sqrt{2}}{2}}t)v_2 ,\sqrt{2}\sinh({ \frac{\sqrt{2}}{2}}t)v_1+\cosh({ \frac{\sqrt{2}}{2}}t)v_2\big)^t
\end{gather*}
 which are all null--vectors. The projective slot vanishes for the claimed $t$ and along the zero locus, we get constant function with value 
$(0,0,0,0,{ \frac{\sqrt{2}}{2}}\sqrt{2v_1^2-v_2^2} ,0)^t$ and thus the closed orbit is of type $M^{0,\perp}$. On the open orbits, we get the claimed Ricci flat metrics. On the closed orbit, it is easy to observe that the annihilator of 
$(0,0,0,0,{ \frac{\sqrt{2}}{2}}\sqrt{2v_1^2-v_2^2} ,0)^t$ is isomorphic to $SO(1,3)\rtimes \R^4$ and that $SO(1,3)\rtimes \R^4\cap P=P_1\rtimes \R^3.$ We consider $w=
(0,0,0,0,{ \frac{\sqrt{2}}{2}}\sqrt{2v_1^2-v_2^2} ,0)^t$ for the geometric interpretation and thus we can use the $\fc$--(co)frame to deduce the underlying geometric structure, where the distribution is given by $e_4$ and the conformal class by the conformal basis $e_2,e_3$.
\end{proof}

Since the normal conformal Killing field can be obtained as a tensor product of two twistor spinors in 
our case, we discuss them together. For the twistor spinors $v\in \mathbb{N}$, we know that $G$--types correspond to the length of $v$ w.r.t. the Hermitian metric on the spin tractor bundle preserved by $SU(2,2)$. In the case of normal conformal Killing fields, the $G$--types coincide with the classification of adjoint orbits in $\so$,  \cite[Table III, 13--16]{Djokovic}.

\begin{prop}\label{pp-spin}
The normal conformal Killing field $u$ from Proposition \ref{pp-sol} has the 
$G$--type corresponding to translations and the Lie algebra $\fg_u$ of $G_u$ consisting of elements 
\begin{align*} 
 \left[ \begin {smallmatrix}h_8&h_6&0&0&0 &0
\\ h_7&-h_8&0 & 0 &0&0
\\ h_3&h_1&0&h_9&0&0
\\ h_4&h_2&-h_9&0&0&0
\\ h_5&0&-h_1&-h_2&h_8&-h_6
\\ 0&-h_5&-h_3&-h_4&-h_7&-h_8
\end {smallmatrix} \right].
\end{align*}
The twistor spinor $v$ from Proposition \ref{pp-sol} has the null $G$--type 
 and the Lie algebra $\fg_v$ of $G_v$ is subalgebra of $\fg_u$ for $h_9=0$.

All points of $M$ have the same $P$--types and there are homogeneous Cartan geometries of type $(G_v,P_v)$ and $(G_u,P_u)$ on $M$, where $P_v$ has the Lie algebra generated by $h_1,h_2,h_6,h_8$--parts of $\fg_v$ and $P_u$ has the Lie algebra generated by $h_1,h_2,h_6,h_8,h_9$--parts of $\fg_u$. In particular, the $\fc$--(co)frame and the restriction of the map $\alpha$ to $\fg_v$ and $\fg_u$ provides the underlying geometric structure consisting of
\begin{enumerate}
\item a distinguished vector field $e_4=\partial_z$, that is a reduction of $CO(1,3)$ to $P_u\cap CO(1,3)=(P_v\cap CO(1,3))\rtimes SO(2)$,
\item an orthogonal decomposition of $e_4^\perp/\langle e_4\rangle=\langle e_2\rangle\oplus \langle e_3 \rangle$ into two distributions of rank $1$, that is a reduction of $(P_v\cap CO(1,3))\rtimes SO(2)$ to $(P_v\cap CO(1,3))$, and
\item a subclass of the class of Weyl connections that preserve $e_4$ and the two distributions of rank $1$ induced by the Levi--Civita connection of $g$, and one--forms proportional to $e^1$ via the usual formula for the change of Weyl connection, \cite[Section 1.6]{parabook}.
\end{enumerate}
\end{prop}
\begin{proof}
It follows from Proposition \ref{pp-sol} that the  twistor spinor $v\in \mathbb{N}$ has the null $G$--type and the normal conformal Killing field can be obtained as tensor product of  two twistor spinors. There are two possible $P$--types corresponding to null vectors in the maximal null--plane given by the first two vectors of the standard (complex) basis of $\C^4$ and to null vectors outside such a maximal null--plane. Since they are constant, we can compute $\fg_w=\fg_v$ and $\fg_w=\fg_u$ for the representative $w=v$ and $w=u$, respectively, i.e., all the points have the same $P$--type. Since $\alpha(\fk)=\alpha(\fk_v)\cap \fg_v=\alpha(\fk_u)\cap \fg_u$, the geometry is homogeneous and we can use the $\fc$--(co)frame to describe it in the claimed way. Indeed, $h_5$ corresponds to $e_4$ and defines a distinguished vector field on $TM$, $h_3,h_4$ correspond to $e_2,e_3$ (modulo $\fp_v$) and provide the orthogonal decomposition of the quotient. Finally, $h_6$ corresponds to $e^1$ and describes the change of the Levi--Civita connections of $g$ to the distinguished subclass of the class of Weyl connections.
\end{proof}

Let us finally remark that we also computed that there is a $27$--dimensional family of conformal Killing $2$--tensors and a $2$--parameter family of conformal Killing--Yano $2$--forms. Therefore, all of the conformal Killing--Yano 2--forms are normal.

In the next, we discuss holonomy reductions on an example that carries no Einstein scales.
In order to have at least one non--reductive example, we consider a non--reductive analog of G\"odel metric.

\subsection{Holonomy reductions on non--reductive analog of G\"odel metrics} \label{nonred-hol}
 We pick an example such that its group of conformal symmetries that has a similar structure as for the G\"odel metrics, however, we replace $SO(2)$ with a non--trivial $1$--dimensional representation $L$ of $\R$ isomorphic to the action of diagonal $Sl(2,\R)$--matrices on strictly upper triangular $Sl(2,\R)$ matrices, i.e., $K=(\R\rtimes L)\times Sl(2,\R)$. So this time we fix $H=\Delta(L)$ to be the diagonal in the product of $L$ and strictly upper triangular matrices in $Sl(2,\R)$. Thus we can parametrize the Lie algebra $\fk$ as 
$$
 \left[ \begin {smallmatrix} { x_3}&{ x_5}&0&0
\\ 0&-{ x_3}&0&0\\ 0&0&-{ x_2}
+{ x_3}&2{ x_4}+{ x_5}-{ x_1}\\ 0&0&{ 
x_1}&{ x_2}-{ x_3}\end {smallmatrix} \right]
$$
with $x_5$ parameterizing $\fh$. Let us show that if we consider the complement $\fc$ parametrized by $(x_1,x_2,x_3,x_4)$, then $\alpha_{-1}: \fc\to \R^4$ given by this parametrization describes a $K$--invariant conformal geometry on the non--reductive homogeneous space $K/H$.

\begin{lem}\label{nonred-coframe}
There is a decomposition $\fk=\fl\oplus \fa\oplus \fn$ such that
\begin{gather*}
e_1:=e^{2x_2-2x_3}\partial_{x_1}-2x_4\partial_{x_2}-{\scriptstyle \frac{1}{2}}(4x_4^2+1-e^{4x_2-4x_3})\partial_{x_4},\ \ 
e_2:= \partial_{x_2}+2x_4\partial_{x_4},
\\
e_3:=\partial_{x_3}-2x_4\partial_{x_4},\ \ 
e_4:=\partial_{x_4}\\
e^1:=e^{-2x_2+2x_3}dx_1,\ \ e^2:=2x_4e^{-2x_2+2x_3}dx_1+dx_2,\ \ e^3:=dx_3,\\
e^4:=({\scriptstyle  \frac{1}{2}}e^{-2x_2+2x_3}-2x_4^2e^{-2x_2+2x_3}-{\scriptstyle \frac{1}{2}}e^{2x_2-2x_3})dx_1-2x_4dx_2+2x_4dx_3+dx_4,
\end{gather*}
are $\fc$--(co)frames in the exponential coordinates $\fc=(x_1,x_2,x_3,x_4)\to K/H$ compatible with the decomposition $\fl\oplus \fa\oplus \fn$, i.e.,
\begin{align*}
g&= (e^{-4x_2+4x_3}-1)dx_1^2+e^{-2x_2+2x_3}(dx_1dx_4+dx_4dx_1)+dx_2^2+dx_3^2+
\\ &\ \ \ \  2x_4e^{-2x2+2x3}(dx_3dx_1+dx_1dx_3)
\end{align*}
is a metric in the $K$--invariant conformal class on $K/H.$
Moreover, 
$$
\alpha(x_1,x_2,x_3,x_4,x_5)= \left[ \begin {smallmatrix} -{ x_2}&4{ x_1}+2{ x_4}+{ 
x_5}&0&0&0&0\\ { x_1}&{ x_2}-2{ x_3}&0&0&0&0
\\ { x_2}&2{ x_1}-2{ x_4}-{ x_5}&0&-{
 x_3}&0&0\\ { x_3}&-2{ x_1}&{ x_3}&0&0&0
\\ { x_4}&0&-2{ x_1}+2{ x_4}+{ x_5}&2{
 x_1}&-{ x_2}+2{ x_3}&-4{ x_1}-2{ x_4}-{ x_5}
\\ 0&-{ x_4}&-{ x_2}&-{ x_3}&-{ x_1}&{ x_2
}\end {smallmatrix} \right] 
$$
is the associated normal conformal extension of $(\fk,\fh)$ with curvature 
$$
\kappa(\alpha(x_i),\alpha(y_i))= \left[ \begin {smallmatrix} 0& 20({ x_1}( { y_2}-{ y_3}
 ) -{ y_1} ( { x_2}-{ x_3} )) &0&0
&0&0\\ 0&0&0&0&0&0\\ 0& -2({ x_1}( 
{ y_2}+3{ y_3} ) -{ y_1} ( { x_2}+
3{ x_3} )) &0&0&0&0\\ 0&  -2({ x_1}( 3{ y_2
}-{ y_3} )- { y_1}( 3{ x_2}-{ x_3}
 ) )&0&0&0&0\\ 0&0& * &*&0& * \\ 0&0&0&0&0&0\end {smallmatrix} \right],
$$
where the $*$--entries are determined by \eqref{blockso}.
\end{lem}
\begin{proof}
It is not hard to check that $g_o=\alpha_{-1}^*\nu_{1,3}$ is an $H$--invariant element of $\bigodot^2 \fk/\fh^*$ and thus defines a $K$--invariant conformal geometry on $K/H$. Further, we  can observe that there is decomposition $\fk=\fl\oplus \fa\oplus \fn$, where $\fl$ is given by $x_1$, $\fa$ is given by $x_2,x_3$ and $\fn$ is given by $x_4,x_5$. Since $\fc \cap \fl$ is given by $x_1$, $\fc \cap \fa$ is given by $x_2,x_3$ and $\fc \cap \fn$ is given by $x_4$, we have the claimed exponential coordinates. The pullback of the Maurer--Cartan form of $K$ to $T\fc$ takes the following form in these coordinates
$$
\left[\begin{smallmatrix}
dx_3&0& 0& 0\\
0&-dx_3& 0& 0\\
0& 0& -2x_4e^{-2x2+2x3}dx_1-dx_2+dx_3& -(4x_4^2e^{-2x2+2x3}+e^{2x2-2x3})dx_1 -4x_4dx_2+4x_4dx_3 + 2dx_4\\
0&0& exp^{-2x2+2x3}dx_1&2x_4e^{-2x2+2x3}dx_1+ dx_2-dx_3.
\end{smallmatrix}\right].
$$
This induces the claimed $\fc$--(co)frame and the metric $g$ in the conformal class according to Proposition \ref{c-coframe}.
With this information, we can directly compute the normal conformal extension $\alpha:\fk \to \sodc$ in the following steps.
\begin{enumerate}
\item We start with the above $\alpha_{-1}$.
\item We find the image $d\iota_0(x_5)$ in $\frak{co}(1,3)$ with the same graded action on $(x_1,x_2,x_3,x_4)$ in $\fc$ and then we compute the full $d\iota(x_5)$ using the conditions from the proof of Proposition \ref{conf_ext}.
\item We compute $\alpha_0(x_1,x_2,x_3,x_4)$ and $\alpha_1(x_1,x_2,x_3,x_4)$ using the  normalization condition \eqref{norm} and we find the least possible $a$--part in \eqref{blockso} of $\alpha_0$  (it is not generally possible to get $a=0$ on non--reductive homogeneous space).
\end{enumerate}
The formula for the curvature follows the definition.
\end{proof}

Further, let us show that this conformal geometry admits normal solutions that we can use for the holonomy reductions.

\begin{prop}\label{nonred-sol}
Suppose  $\rho=w_1\lambda_1+w_2\lambda_2+w_3\lambda_3$  for  fundamental weights $\lambda_i$ of complexification of $\sodc$. 
\begin{itemize}
\item If $w_1\neq 0$, then the corresponding first BGG operator does not have any normal solutions. 
\item If $w_1=0$ and $w_2\neq w_3$, then there is a $2$--parameter family of normal solutions with $$\Phi(x_1,x_2,x_3,x_4)(v_1+iv_2)=\big((w_2+w_3)x_3+\frac12(w_2-w_3)ix_3\big)(v_1+iv_2)$$ and $\pi: \mathbb{N}\to \mathbb{X}$ identifying $v_1+iv_2$ with $(v_1+iv_2)$--multiple of the lowest weight vector in $\mathbb{X}.$ 
\item 
If $w_1=0$ and $w_2= w_3$, there is a $1$--parameter family of normal solutions with $\Phi(x_1,x_2,x_3,x_4)v_1=2w_2x_3v_1$ and $\pi: \mathbb{N}\to \mathbb{X}$ identifying $v_1$ with $v_1$--multiple of the lowest weight vector in $\mathbb{X}.$
\end{itemize}
In particular, there is a $2$--parameter family of twistor spinors given by function $$s: \fc\to \mathbb{X}=\mathbb{C}^2[{\scriptstyle \frac12}],\ \ \ s(x_1,x_2,x_3,x_4):=(0,e^{-x_3(1+\frac12i)}(v_1+iv_2))^t$$ and there is a $1$--parameter family of normal conformal Killing vectors $v_1e^{-2x_3}\partial_{x_4}$.
\end{prop}
\begin{proof}
We start by computation of the infinitesimal holonomy $hol(\alpha)$. Firstly, the image of curvature is the $2$--dimensional space of elements
$$
 \left[ \begin {smallmatrix} 0&2h_1-4h_2&0&0&0 &0
\\ 0&0&0 &0 &0&0
\\ 0&h_1&0&0&0 &0
\\ 0&h_2&0&0&0&0
\\ 0&0&-h_1&-h_2&0&-2h_1+4h_2\\ 
0&0&0&0&0&0
\end {smallmatrix} \right].
$$
Then the bracket of the image of the curvature with the image of $\alpha$ provides $6$--dimensional space of elements
$$
 \left[ \begin {smallmatrix} 2h_3-4h_4&h_6&0&0&0 &0
\\ 0&-2h_3+4h_4&0 & 0 &0&0
\\ h_3&h_1&0&0&0 &0
\\ h_4&h_2&0&0& 0&0
\\ h_5&0&-h_1&-h_2&2h_3-4h_4&-h_6
\\ 0&-h_5&-h_3&-h_4&0&-2h_3+4h_4
\end {smallmatrix} \right].
$$
The next bracket of the above $6$--dimensional space with the image of $\alpha$ provides $8$--dimensional space of elements
$$
 \left[ \begin {smallmatrix}h_8&h_6&0&0&0 &0
\\ h_7&-h_8&0 & 0 &0&0
\\ h_3&h_1&0&0&0&0
\\ h_4&h_2&0&0&0&0
\\ h_5&0&-h_1&-h_2&h_8&-h_6
\\ 0&-h_5&-h_3&-h_4&-h_7&-h_8
\end {smallmatrix} \right].
$$
The next bracket does not provide anything new and thus we have got $8$--dimensional $hol(\alpha)$. Let us emphasize that $h_1,h_2,h_3,h_4,h_5$ form the Heisenberg Lie algebra corresponding to a contact grading of $\sodc$ and $h_6,h_7,h_8$ form $\frak{sl}(2,\R)$.

Now, we can use the representation theory to deduce all the normal solutions. Since $hol(\alpha)$ contains all negative root spaces, the normal solutions are precisely the $\frak{so}(2)$--modules of the lowest weight on which $h_8$ acts trivially, which provides the projection $\pi$. If the highest weight is $w_1\lambda_1+w_2\lambda_2+w_3\lambda_3$, then the lowest weight is $-w_1\lambda_1-w_3\lambda_2-w_2\lambda_3$ and $(-w_1\lambda_1-w_3\lambda_2-w_2\lambda_3)(h_8)=h_8(-w_1)$. Thus the existence claims on the normal solutions follow, because the $\frak{so}(2)$--module of the lowest weight is complex for $w_2\neq w_3$ and is real for $w_2= w_3$. We can then use the conformal extension $\alpha$ to obtain the representation $\Phi$.

The claims on the twistor spinors and normal conformal Killing vector fields can also be observed from the image
\[
 \left[ \begin {smallmatrix} 0&0&0&0\\ 
-\frac{\sqrt {2}}{8}(h_1+ih_2)&h_8&-\frac14h_6&0\\ 
-{ \frac{\sqrt {2}}{2}}(h_3-i h_4)&-4h_7&-h_8&0\\ 
\frac14 h_5&-{ \frac{\sqrt {2}}{2}}(h_3+i h_4)&\frac{\sqrt {2}}{8}(h_1-ih_2)&0\\ 
\end {smallmatrix} \right]
\]
of $hol(\alpha)$ under the spin representation of $\sodc$ and from the fact that the normal conformal Killing field can be also seen as a tensor product of two twistor spinors. 
\end{proof}

Let us discuss the holonomy reductions induced by the twistor spinors and normal conformal Killing fields from Proposition \ref{nonred-sol}.

\begin{prop}
The twistor spinor $v$ from Proposition \ref{nonred-sol} has the null $G$--type, and the  normal conformal Killing field $u$ from Proposition \ref{nonred-sol} has the $G$--type corresponding to translations. In particular, $\fg_v,\fg_u,P_v,P_u$ are as in Proposition \ref{pp-spin}.

 All the points of $K/H$ have the same $P$--type and there are Cartan geometries of type $(G_v,P_v)$ and $(G_u,P_u)$ on $K/H$. The Cartan geometry of type $(G_v,P_v)$ is not homogeneous and its symmetry group has orbits of codimesion $1$. The Cartan geometry of type  $(G_u,P_u)$ is homogeneous. The underlying geometric structure consists of
\begin{enumerate}
\item a choice of null--vector $U:=e^{-2x_3}\partial_{x_4}$ that is a reduction of $CO(1,3)$ to $(P_u\cap CO(1,3)=(P_v\cap CO(1,3))\rtimes SO(2)$,
\item an orthogonal decomposition of 
\begin{gather*}
U^\perp/\langle U\rangle = \langle \cos(x_3)(\partial_{x_2}+2x_4\partial_{x_4})-\sin(x_3)(\partial_{x_3}-2x_4\partial_{x_4})\rangle\oplus 
\\
 \langle \sin(x_3)(\partial_{x_2}+2x_4\partial_{x_4})+\cos(x_3)(\partial_{x_3}-2x_4\partial_{x_4})\rangle
\end{gather*}
 into two distributions of rank $1$ that is reduction of $(P_v\cap CO(1,3))\rtimes SO(2)$ to $(P_v\cap CO(1,3))$, and
\item a subclass of the class of Weyl connections that preserve $U$ and the two  distributions of rank $1$ formed by Levi--Civita connections of the metrics $e^{2x_2+f(x_1)}g$ in the conformal class for arbitrary function $f$. Their Ricci tensor vanishes on insertion of vectors in $U^\perp$ as a consequence of normality of the solutions.
\end{enumerate}
\end{prop}
\begin{proof}
It follows from Proposition \ref{nonred-sol} that the situation is similar to the situation in Proposition \ref{pp-spin}.
 The difference is that the section $(0,0,0,e^{-x_3(1+\frac12i)}v)^t$ of the spin tractor bundle is not constant. 
We compute that the element $p_{x_3,w}$ given by exponential of the element of $\frak{co}(1,3)$ with $a_{11}=a=-x_3,$ $b_{12}=-x_3$ in \eqref{blockso} and remaining parts vanishing normalizes the sections to $w=v$ or $w=u$ from Proposition \ref{pp-spin}, respectively, i.e., all the points of $M$ have the same $P$--type. 
Thus by Proposition \ref{hol-red}, there is a Cartan geometry of type $(G_v,P_v)$ and $\Ad_{p_{x_3,w}}^{-1}(\alpha(\fk))\cap \fg_v$ is a subalgebra of $\fk$ with $x_3=0$ and thus has the action of codimension $1$ on $M$. Similarly, there is a Cartan geometry of type $(G_u,P_u)$ and $\Ad_{p_{x_3,w}}^{-1}(\alpha(\fk))\cap \fg_v\cong \fk$ and the Cartan geometry is homogeneous. 
For the interpretation of the underlying geometric structure, it is not hard to deduce that the objects (1),(2), and (3) are equivalent to the claimed reductions. 
Indeed, the coframe corresponding to the reduction is obtained from our coframe $e^i$ by the adjoint action of $p_{x_3,w}^{-1}$ and has the claimed form.
\end{proof}

To highlight the similarity with the conformal class of G\"odel metrics, we describe all conformal Killing $2$--tensors and conformal Killing--Yano $2$--forms using the data computed according to Theorem \ref{thm-loc}. The explicit formulas for the conformal Killing $2$--tensors and conformal Killing--Yano $2$--forms can be obtained using the formula \eqref{parsec} and the $\fc$--(co)frame from Lemma \ref{nonred-coframe}.

\begin{prop}
There is a $15$--dimensional space of conformal Killing $2$--tensors that decomposes into the following $K$--invariant families. We use the para\-met\-ri\-zation $(s_1,\dots,s_9)$ of the projective slot given by the $\fc$--frame
\begin{gather*}
\bigodot{}^2_0TM=\{ s_1 e_1^2+s_2e_1e_2+s_2e_2e_1+s_3e_1e_3+s_3e_3e_1+s_4e_1e_4+s_4e_4e_1\\
+s_5e_2^2+s_6e_2e_3+s_6e_3e_2+s_7e_2e_4+s_7e_4e_2\\
- (2s_4+s_5)e_3^2+s_8e_3e_4+s_8e_4e_3+s_9e_4^2 \}
\end{gather*}
for the particular components of the projection $\pi:\mathbb{S} \to \bigodot{}^2_0TM$.
\begin{enumerate}
\item There is a $1$--dimensional trivial $\fk$--representation $\R=\langle v_1 \rangle$ with projection 
$$(0, 0, 0, v_1, v_1, 0, 0, 0, 4v_1).$$
\item There is a $3$--dimensional representation $\R^3=\langle v_2,v_3,v_4 \rangle$ that is trivial as the representation of  $\frak{sl}(2,\R)$, while the radical acts by the matrix
$
\left[ \begin{smallmatrix}
4x_3& x_5& 0\\
0& 2x_3& x_5\\
0& 0& 0
\end{smallmatrix} \right]
$
with projection
$$(0, 0, 0, -2v_4, -v_4, v_4, v_3, v_3, -6v_4+2v_2),$$
where $v_2$ corresponds to the normal solutions that are symmetrized products of normal conformal Killing vectors with themselves.
\item There is a $5$--dimensional $\frak{sl}(2,\R)$--representation $$\bigodot^4 \R^{2*}=\{v_5 y_1^4+v_6 y_1^3y_2+v_7y_1^2y_2^2+v_8y_1y_2^3+v_9y_2^4\}$$ with trivial action of the radical and with projection
$$(24v_5, 6v_6, 0, -2v_7+6v_5, 4v_7-6v_5, 0, -3v_8+3v_6, 0, 6v_9-2v_7+6v_5).$$
\item  There is a $6$--dimensional $\frak{sl}(2,\R)$--representation consisting of two copies of $\bigodot^2 \R^{2*}$ that we parametrize as $v_{10} y_1^2+v_{11} y_1y_2+v_{12}y_2^2$ and $v_{13} y_1^2+v_{14} y_1y_2+v_{15}y_2^2$ with the action $
\left[ \begin{smallmatrix}
2x_3& x_5\\
 0& 0
\end{smallmatrix} \right]
$ of radical intertwining these two copies and with projection
\begin{gather*}(0, -4v_{13}, -4v_{13}, v_{14}-2v_{10}, -3v_{14}+2v_{10}, -2v_{14}, 
\\
2v_{15}-2v_{13}-2v_{11}, 2v_{15}-2v_{13}, 4v_{12}-4v_{10}).
\end{gather*}
\end{enumerate}

There is a $2$--dimensional space of  conformal Killing--Yano $2$--forms that are all normal and can be obtained as symmetrized products of twistor spinors with themselves.
\end{prop}
\begin{proof}
We present here only the minimal set of data from Theorem \ref{thm-loc} that were computed using Maple.  
We know that in the case of conformal Killing $2$--tensors, $w_1=0$ and $w_2=w_3=2$ and thus there is a  $1$--parameter family of normal conformal  Killing $2$--tensors according to Proposition \ref{nonred-sol}. Similarly in the case of conformal Killing--Yano $2$--forms, $w_1=0$ and $w_2\neq w_3$ and thus, there is a $2$--parameter family of normal conformal Killing--Yano $2$--forms. 
\end{proof}

\subsection{Example related to CR geometry}

Let us consider the manifold $M$ to be the Lie group $K=Gl(2,\R)$. Consider the $K$--invariant metric defined by $$\alpha_{-1}\left( \left[ \begin {smallmatrix} \frac{1}{12}{ x_4}+2{ x_1}&\frac{\sqrt {3}}{6}
 ( 3{ x_2} -{ x_3} ) \\ \frac{\sqrt 
{3}}{6} ( 3{ x_2}+{ x_3} ) &{\frac {7}{12}}{ x_4}+2
{ x_1}\end {smallmatrix} \right] \right)=(x_1,x_2,x_3,x_4)^t.
$$
We use  Proposition \ref{conf_ext} to compute the corresponding conformal extension $\alpha$ and the infinitesimal holonomy $hol(\alpha)$.

\begin{lem}
The normal conformal extension $\alpha: \frak{gl}(2,\R)\to\sodc$ for the above $\alpha_{-1}$ and the infinitesimal holonomy are as follows
$$\alpha_{-1}( \left[ \begin {smallmatrix} \frac{1}{12}{ x_4}+2{ x_1}&\frac{\sqrt {3}}{6}
 ( 3{ x_2} -{ x_3} ) \\ \frac{\sqrt 
{3}}{6} ( 3{ x_2}+{ x_3} ) &{\frac {7}{12}}{ x_4}+2
{ x_1}\end {smallmatrix} \right] )= \left[ \begin {smallmatrix} 0&-{ x_1}&{ x_2}&-\frac23{ x_3}&\frac23
{ x_4}&0\\ { x_1}&0&\frac23{ x_3}&{ x_2}&0&-
\frac23{ x_4}\\ { x_2}&-{ x_3}&0&-{ x_1}-\frac12
{ x_4}&-\frac23{ x_3}&-{ x_2}\\ { x_3}&{ x_2
}&{ x_1}+\frac12{ x_4}&0&-{ x_2}&\frac23{ x_3}
\\ { x_4}&0&{ x_3}&-{ x_2}&0&{ x_1}
\\ 0&-{ x_4}&-{ x_2}&-{ x_3}&-{ x_1}&0
\end {smallmatrix} \right],
$$
$$hol(\alpha)=\left[ \begin {smallmatrix} { h_5}&-{ h_6}&{ h_2}&-{ h_1}&{
 h_8}&0\\ { h_6}&{ h_5}&{ h_1}&{ h_2}&0&-{
 h_8}\\ { h_3}&-{ h_4}&0&2{ h_6}&-{ h_1}
&-{ h_2}\\ { h_4}&{ h_3}&-2{ h_6}&0&-{ 
h_2}&{ h_1}\\ { h_7}&0&{ h_4}&-{ h_3}&-{ 
h_5}&{ h_6}\\ 0&-{ h_7}&-{ h_3}&-{ h_4}&-{
 h_6}&-{ h_5}\end {smallmatrix} \right] \cong \frak{su}(2,1).
$$
\end{lem}
\begin{proof}
Since $\fk=\fc$, the construction from Proposition \ref{conf_ext} simplifies to finding the maps $\alpha_0,\alpha_1$ that are uniquely determined in the given form by the normalization conditions \eqref{norm} and vanishing of the $a$--part in \eqref{blockso} of $\alpha_0$.
Thus the curvature $\kappa$ takes form
$$\left[ \begin {smallmatrix} 0&0&-(\frac56{ x_3}{ y_4}-\frac56{ 
x_4}{ y_3})&-(\frac56{ x_2}{ y_4}-{ x_4}{ y_2})&0&0
\\ 0&0&\frac56({ x_2}{ y_4}-{ x_4}{ 
y_2})&-\frac56({ x_3}{ y_4}-{ x_4}{ y_3})&0&0
\\ 0&0&0&0&-\frac56({ x_2}{ y_4}-{ x_4}{
 y_2})&\frac56({ x_3}{ y_4}-{ x_4}{ y_3})
\\ 0&0&0&0&\frac56({ x_3}{ y_4}-{ x_4}{
 y_3})&\frac56({ x_2}{ y_4}-{ x_4}{ y_2})
\\ 0&0&0&0&0&0\\ \noalign{\medskip}0&0&0&0&0&0
\end {smallmatrix} \right] $$
and its image corresponds to $h_1,h_2$--entries of $hol(\alpha).$ In two steps, the bracket of the image of $\kappa$ with the image of $\alpha$ generates the given $hol(\alpha)$. Simple analysis of this subalgebra of $\sodc$ shows that it is a simple Lie algebra $\frak{su}(2,1).$
\end{proof}

Consequently, normal solutions correspond to trivial $\frak{su}(2,1)$-submodules in the branching of the representation $\rho: \sodc\to \frak{gl}(V) $ to $\frak{su}(2,1)$. In particular, we have the following normal conformal Killing field inducing a holonomy reduction to a CR geometry.

\begin{prop}
The center of $\frak{gl}(2,\R)$ is generated by a normal conformal Killing field and $Gl(2,\R)/\R$ carries an $Sl(2,\R)$--invariant CR geometry with CR--distribution generated by $x_2,x_3$--entries of $\frak{gl}(2,\R)$ and complex structure identifying it with $x_2+ix_3.$
\end{prop}
\begin{proof}
The adjoint representation of $\sodc$ contains a single trivial $\frak{su}(2,1)$--submodule corresponding to the $x_1$--entry of the image of $\alpha$ and thus, the center of $\frak{gl}(2,\R)$ is generated by a normal conformal Killing field.  The entries $h_3,h_4,h_7$ clearly determine a negative part of the grading of $\frak{su}(2,1)$ corresponding to the CR geometry, and comparison with the image of $\alpha$ determines the claimed description of the CR geometry.
\end{proof}

Note that this is (up to covering) an example of a CR geometry contained in a $1$--parameter class of symmetric CR geometries on $SO_0(2,1)$ (for the value of the parameter $t=\sqrt{3}$)  constructed in \cite{ja-DGA}.

\section{Applications of all solutions, conserved quantities and conformal circles}\label{section5}

\subsection{Conformal circles via conserved quantities} 
There is a distinguished family of curves defined on each conformal geometry called \emph{conformal circles} generalizing geodesics of Riemannian geometry, \cite{BE}.
Let $\gamma: \R\to \fc$ be a curve on $(M=K/H,[g])$ in some exponential coordinates ${\sf c}: \fc\to M$. In this Section, we consider more general exponential coordinates than those compatible with decomposition $\fl\oplus \fa \oplus \fn$, because it will simplify some computations. According to \cite{SZ,conserved}, we assign to nowhere null curve $\gamma$, where we always omit writing
 the argument $t$ for $\gamma$ and its components in the local coordinates, a conformally invariant curve $\Sigma^{\gamma}:\R\to \wedge^3 \mathbb{T}$ as follows. For a chosen $g$ in the conformal class, we denote $$u:=\|\gamma'\|=\sqrt{g(\gamma',\gamma')}$$ and we can define three curves $X^\gamma,U^\gamma,A^\gamma: \R\to \mathbb{T}$ as
\begin{gather*}
X^\gamma:=\left[ \begin{smallmatrix}
\frac{1}{u}\\
0\\
0
\end{smallmatrix} \right],\\
U^\gamma:=(X^\gamma)'+\rho\Big(\alpha  \big((\alpha_{-1})^{-1}  \big(e^1(\gamma'), \dots, e^n(\gamma')\big)+\sum_i H_ie^i (\gamma') \big)\Big) (X^\gamma),\\
A^\gamma:=(U^\gamma)'+\rho\Big(\alpha \big((\alpha_{-1})^{-1}\big(e^1(\gamma'), \dots, e^n(\gamma')\big)+\sum_i H_ie^i (\gamma') \big)\Big) (U^\gamma),
\end{gather*}
where $\rho(\alpha \circ((\alpha_{-1})^{-1}\circ (e^1, \dots, e^n)+\sum_i H_ie^i))$ is the algebraic part of the tractor derivative along the curve $\gamma$ expressed according to Proposition \ref{localext} using the $\fc$--coframe $ (e^1, \dots, e^n)$. Then $$\Sigma^{\gamma}:=X^\gamma\wedge U^\gamma \wedge A^\gamma.$$ Conformal circles are characterized by the property, \cite{SZ},
\[
\Sigma^\gamma\wedge \bigg( (A^\gamma)'+\rho\Big(\alpha \big((\alpha_{-1})^{-1}\big(e^1(\gamma'), \dots, e^n(\gamma')\big)+\sum_i H_ie^i (\gamma') \big)\Big) (A^\gamma)\bigg) =0.
\]

The basic result from \cite{conserved} on conserved quantities along the conformal circles we consider in this article is the following.

\begin{prop}
Let $s: \fc\to \wedge^3 \mathbb{T}$ be the function corresponding to a Killing--Yano $2$--form. Then $\bg(s,\Sigma^\gamma)$ is constant along conformal circle $\gamma$, where $\bg$ is the induced tractor metric on $\wedge^3 \mathbb{T}$.
\end{prop}

Thus if we have enough Killing--Yano $2$--forms, then we can use the corresponding conserved quantities to describe the conformal circles. Note that the representation of $K$ on the functions corresponding to Killing--Yano $2$--forms induces a representation on the conserved quantities. One can then use this action to set the conserved quantities into a particular form and obtain some representatives of the $K$--orbits of conformal circles. Let us illustrate this with examples. We start with the conformal structure on $3$--dimensional Heisenberg group which is studied in \cite{Tod} under the name $Nil$.

\subsection{Conformal circles on $3$--dimensional Heisenberg group}

We consider the group $K$ of lower triangular $3\times 3$--matrices with ones on the diagonal
\[
K=\left[ \begin{smallmatrix}
1&0 & 0\\
x_2&1 & 0\\
x_3& x_1 & 1
\end{smallmatrix} \right].
\]
We define a conformal class on $M=K$ by choosing the following $\fc$--coframe
\[
e^1:=\frac1{\sqrt{2}}(dx_1-dx_2),\ \ e^2:=dx_3-x_1dx_2,\ \ e^3:=\frac1{\sqrt{2}}(dx_1+dx_2),
\]
i.e., $g=dx_1^2-dx_2^2+(dx_3-x_1dx_2)^2$ determines a conformal class of signature $(1,2)$ and $e^2$ is the usual contact form on the Heisenberg group. 

Note that in comparison with \cite{Tod}, we have different conventions and thus the conserved quantities from \cite[Section 4.2]{Tod} become
\[
E=2a_1a_3+a_2^2+u_2^2,\ \  J=-a_3u_1+a_1u_3-\frac12 u_2,
\]
where 
$$\gamma(t)=\left[ \begin{smallmatrix}
0&0& 0\\
\gamma_2(t)& 0 & 0 \\
\gamma_3(t)& \gamma_1(t) & 0
\end{smallmatrix} \right]\in \fc \simeq \fk$$
 is a curve parametrized by arc--length w.r.t. the metric $g$ and
\[
X^\gamma:=\left[ \begin{smallmatrix}
1\\
0\\
0\\
0\\
0
\end{smallmatrix} \right],U^\gamma:=\left[ \begin{smallmatrix}
0\\
u_1:=\frac1{\sqrt{2}}(\gamma_1'-\gamma_2')\\
u_2:=\gamma_3'-\gamma_1\gamma_2'\\
u_3:=\frac1{\sqrt{2}}(\gamma_1'+\gamma_2')\\
0
\end{smallmatrix} \right],A^\gamma:=\left[ \begin{smallmatrix}
u_2^2-\frac38\\
a_1:=-u_1u_2+u_1'\\
a_2:=u_2'\\
a_3:=u_3u_2+u_3'\\
-1
\end{smallmatrix} \right].
\]
In such a non--holonomic coordinates, we have the following equations for the conformal circles
\begin{gather*}
a_1' = -u_1(2a_1a_3+a_2^2)+u_1u_2^2+\frac12u_1a_2+\frac12u_2a_1,\\
a_2' = -u_2(2a_1a_3+a_2^2)+u_2^3-u_2+\frac12u_3a_1-\frac12u_1a_3,\\
a_3' = -u_3(2a_1a_3+a_2^2)+u_3u_2^2-\frac12u_3a_2-\frac12u_2a_3,\\
u_1a_3+u_2a_2+u_3a_1=0,
\end{gather*}
where the last equations is a consequence of the assumption $2u_1u_3+u_2^2=1$.

Let us compute the conserved quantities provided by the conformal Killing--Yano $2$--forms.

\begin{prop}
The following quantities are constant along the conformal circles on $(M,[g])$
\begin{gather*}
C_1:={\frac{2\sqrt {2}}{3}}({2}u_3a_{{1}} -{2}u_1a_{{3}}-u_2
 )( \gamma_{{1}}+ \gamma_{{2}})+{\frac43}(u_3-{2}u_3a_{{2}}+{2}
u_2a_{{3}}),\\
C_2:={\frac15}(u_3a_1-u_1a_3-u_2)(\gamma_1^2-\gamma_2^2)+{\frac{\sqrt {2}}{5}}   \big( 2u_2( a_{{1}}+a_{{3}} )-2a_{{2}} ( u_1+u_3 ) -
\\
u_1+u_3
 \big) \gamma_{{1}}+ {\frac{\sqrt {2}}{5}}\big(2u_2( a_{{1}}-a_{{3}} ) + 2a_{{2}}( u_3-u_1) -u_1-u_3
 \big) \gamma_{{2}} 
 +{\frac85}(u_3a_{{1}}-u_1a_{{3}}),\\
C_3:={\frac14}( {2}u_3a_{{1}}-u_2-{2}u_1a_{{3}} )  ( \gamma_{{1
}}^{2}+\gamma_{2}^{2} ) -{\frac{\sqrt {2}}{4}} \big(  2( u_1
+u_3 ) a_{{2}}-2 ( a_{{1}}+a_{{3}} ) u_2+
\\
u_1-u_3
 \big) \gamma_{{1}}-
{\frac{\sqrt {2}}{4}} \big(  2( u_1-u_3 ) 
a_{{2}}+
 2( a_{{3}}-a_{{1}} ) u_2+u_1+u_3 \big) \gamma_{{2}}+u_2,
\\
C_4:={\frac{2\sqrt {2}}{3}} ( 2u_1a_{{3}}-2u_3a_{{1}}+u_2 )  ( \gamma_{{1}}-
\gamma_{{2}} ) +{\frac43}(2u_1a_{{2}}+u_1-{2}u_2a_{{1}})
\end{gather*}
\end{prop}
\begin{proof}
Let us firstly collect the data necessary for computation of the Killing--Yano $2$--forms. The first ingredient is the conformal extension 
\[\alpha: \left[ \begin{smallmatrix}
0&0 & 0\\
\frac1{\sqrt{2}}(-y_1+y_3)&0 & 0\\
y_2& \frac1{\sqrt{2}}(y_1+y_3) & 0
\end{smallmatrix} \right]\mapsto \left[ \begin{smallmatrix}
 0&-{\frac38}{ y_3}&{\frac58}{ y_2}&-{\frac38}{
 y_1}&0\\ { y_1}&-{\frac12}{ y_2}&-{\frac12}{ y_1}&0
&{\frac38}{ y_1}\\ { y_2}&-{\frac12}{ y_3}&0&{\frac12}{
 y_1}&-{\frac58}{ y_2}\\ { y_3}&0&{\frac12}{ y_3}&{\frac12}{ y_2}&{\frac38}{ y_3}\\ 0&-{ y_3}&-{ y_2}&-{
 y_1}&0
\end{smallmatrix} \right],\]
where the parametrization of $\fk$ is given by evaluation of the dual conformal frame to $e^1,e^2,e^3$ at the origin $o=(0,0,0)$ and then the map $\alpha$ is computed according to Proposition \ref{conf_ext}. The coordinates $(x_1,x_2,x_3)$ are not the exponential coordinates corresponding to this parametrization and comparing them we obtain the transition
\[
y_1 = \frac{\sqrt{2}}{2}(x_1-x_2),\ \  y_2 = -\frac12 x_1x_2+x_3,\ \  y_3 = \frac{\sqrt{2}}{2}(x_1+x_2).
\]
This allows us to present the conserved quantities in $(x_1,x_2,x_3)$--coordinates.

Now we can compute the characterization of conformal Killing--Yano $2$--forms from Theorem \ref{thm-loc}. In particular, we compute that $\mathbb{S}\subset \wedge^3 \mathbb{T}$ has coordinates $$(v_1, v_2, \frac43v_1, v_4, v_3, \frac43v_4, 0, -\frac83v_1, \frac85v_2, -\frac83v_4)$$ in the standard basis $t_{[ijk]}, 1\leq i<j<k\leq 5$ of $\wedge^3\mathbb{T}$ (ordered lexicographically, i.e., $v_3$ corresponds to multiples of $t_{[135]},$) and 
\begin{gather*}
\Phi(y_1,y_2,y_3)(v_1, v_2, \frac43v_1, v_4, v_3, \frac43v_4, 0, -\frac83v_1, \frac85v_2, -\frac83v_4)=\big(-\frac{3}{10}v_2y_1-{\frac38}v_3y_1,\\
 -{\frac53}v_1y_3+{\frac53}v_4y_1, -{\frac12}v_3y_1-{\frac25}v_2y_1,
 \frac{3}{10}v_2y_3+{\frac38}v_3y_3, {\frac43}v_1y_3-{\frac43}v_4y_1,\\ 
{\frac25}v_2y_3+{\frac12}v_3y_3, 0, v_3y_1+\frac45v_2y_1, -{\frac83}v_1y_3+{\frac83}v_4y_1, -v_3y_3-\frac45v_2y_3\big).
\end{gather*}

Then we can extend the elements of $\mathbb{S}$ to sections of the tractor bundle in $(x_1,x_2,x_3)$--coordinates. Finally, we obtain the claimed conserved quantities by contracting them with $\Sigma^\gamma$ using the tractor metric induced by $\bg$.
\end{proof}

We can now analyze all the equations together and find out that we obtain only two new conserved quantities. However, they are enough to simplify the equations for the conformal circles to the following form
\begin{gather*}
\gamma_1 =\frac{\sqrt{2}(6u_3 ( u_2^{2}-1 )  ( C_{{1}}-C_{{4}} ) -
 ( 2u_3^{2}+u_2^{2}-1 )  ( u_2 ( 5C_{{2}}-4C_
{{3}} ) +8u_2^{2}-4-S ))}{4u_3(5C_2-4C_3)(u_2^2-1)},
\\
 \gamma_2 = \frac{\sqrt{2}(6u_3 ( u_2^{2}-1 )  ( C_{{1}}+C_{{4}} ) -
 ( 2u_3^{2}-u_2^{2}+1 )  ( u_2 ( 5C_{{2}}-4C_
{{3}} ) +8u_2^{2}-4-S ) )}{4u_3(5C_2-4C_3)(u_2^2-1)}
,\\
 \gamma_3' = \frac{1}{8u_3^2(5C_2-4C_3)(u_2^2-1)} \big( 6u_3 ( u_2^{2}-1 )  ( 2u_3^{2}+u_2^{2}-1 ) 
 ( C_{{1}}-C_{{4}} ) -
4 ( 2u_2^{2}-1 ) \cdot
\\
 ( 2u_3^{2}+u_2^{2}-1 ) ^{2}-u_2 ( 2u_3^{2}-u_2^{2}+
1 ) ^{2} ( 5C_{{2}}-4C_{{3}} ) +S ( 2u_3^{
2}-u_2^{2}+1 )  ( 2u_3^{2}+u_2^{2}-1 ) \big),
\\
  u_3' = {\frac { ( -8u_2^{3}-Su_2+12u_2+5C_{{2}}-4C_{{3}}
 ) u_3}{8(u_2^{2}-1)}}, \ \ \ 
 u_2' = -\frac{S}{8},\\
 S^2= 64u_2^{4}+ ( 18C_{{1}}C_{{4}}+ ( 5C_{{2}}+12C_{{3}}
 )  ( 5C_{{2}}-4C_{{3}} ) -64 ) u_2^{2}+
\\
8 ( 5C_{{2}}-4C_{{3}} ) u_2-18C_{{1}}C_{{4}}-16C_{{3}}
 ( 5C_{{2}}-4C_{{3}} ) +16.
 \end{gather*}
The particular solution from \cite{Tod} corresponds to the special case with 
$$J = \frac58C_2-\frac12C_3=0, \ \ E = \frac14-\frac{9}{32}C_1C_4-\frac{1}{64}(5C_2+12C_3)(5C_2-4C_3)=0
$$
 and some other special choices can also simplify the equations so they can be solved explicitly.

\subsection{Conformal circles on split version of Fubini--Study metric}

The Fubini--Study metrics are the unique (up to a constant multiple) invariant metrics on the symmetric space $SU(1+n)/U(n)$ and conformal circles on this space are studied in \cite{DT}. We consider the split version of this space that takes the split real form as the isometry group, so we consider the unique (up to a constant multiple) invariant split signature metrics on the symmetric space $Sl(1+n,\R)/Gl(n,\R).$ We assume here $n>1.$

\begin{lem}
The normal conformal extension $\alpha: \frak{sl}(n+1)\to   \frak{so}(n+1,n+1)$ corresponding to the split version of the Fubini--Study metric is
\[
\alpha(\left[ \begin{smallmatrix}
-tr(A)& X_2^t\\
X_1& A
\end{smallmatrix} \right])=
\left[ \begin{smallmatrix}
0& \frac{n+1}{2(2n-1)}X_2^t & \frac{n+1}{2(2n-1)}X_1^t & 0 \\
X_1& A+tr(A)\id & 0 & - \frac{n+1}{2(2n-1)}X_1^t  \\
X_2& 0 & -A^t-tr(A)\id & -\frac{n+1}{2(2n-1)}X_2^t \\
0 & -X_2^t& -X_1^t & 0
\end{smallmatrix} \right]
\]
with curvature
\begin{gather*}
\kappa(\alpha\left[ \begin{smallmatrix}
-tr(A)& X_2^t\\
X_1& A
\end{smallmatrix} \right],\alpha\left[ \begin{smallmatrix}
-tr(B)& Y_2^t\\Y_1& B
\end{smallmatrix} \right])=
\\
\left[ \begin{smallmatrix}
0& 0 &0 & 0 \\
0& \frac{n-2}{2n-1}(X_2Y_1^t-Y_2X_1^t)+(X_2^tY_1-Y_2^tX_1)\id & \frac{n+1}{2n-1}(X_1Y_1^t-Y_1X_1^t) & 0  \\
0&  \frac{n+1}{2n-1}(X_2Y_2^t-Y_2X_2^t) & * & 0 \\
0 &0& 0 & 0
\end{smallmatrix} \right].
\end{gather*}
\end{lem}
\begin{proof}
Since we have an invariant conformal structure on a symmetric space, we can use the general result from \cite{HG-ext} on the structure of the map $\alpha$, and the only missing component we need to compute is the part $\alpha_1.$ It is not hard to verify that the claimed $\alpha_1$ provides the given curvature and that it satisfies the normalization conditions \eqref{norm}.
\end{proof}

Let us start with the case $n=2$ and consider the exponential coordinates
$$ {\sf c}: \left[ \begin{smallmatrix}
0& x_3& x_4\\
x_1& 0 & 0 \\
x_2& 0 & 0
\end{smallmatrix} \right]\mapsto \left[ \begin{smallmatrix}
1&0& 0\\
x_1& 1 & 0 \\
x_2& 0 & 1
\end{smallmatrix} \right] \left[ \begin{smallmatrix}
1& x_3& x_4\\
0& 1 & 0 \\
0& 0 & 1
\end{smallmatrix} \right]o.
$$
Firstly, we find the metric, the $\fc$--coframe, and the pullback of the Cartan connection from Proposition \ref{localext} in our exponential coordinates.

\begin{lem} \label{pfs-lem}
In the above exponential coordinates, we  have $\fc$--coframe
\[
e^1=dx_1,\ \ e^2=dx_2,\ \ e^3=dx_3-x_3x_4dx_2-x_3^2dx_1,\ \ e^4=dx_4-x_3x_4dx_1-x_4^2dx_2
\]
and thus
\[
g= dx_1dx_3+dx_3dx_1+dx_2dx_4+dx_4dx_2-2(x_3dx_1+x_4dx_2)^2.
\]
The component $\alpha(\sum_i He^i)$ is as follows, where $*$--entries depend on the others
\[\left[ \begin{smallmatrix}
0& 0& 0& 0& 0& 0\\
0& 2x_3dx_1+x_4dx_2& x_4dx_1& 0& 0& 0\\
0& x_3dx_2& x_3dx_1+2x_4dx_2& 0& 0& 0\\
0& 0& 0& *& *& 0\\
0& 0& 0& *& *& 0\\
0& 0& 0& 0& 0& 0
\end{smallmatrix} \right].
\]

\end{lem}
\begin{proof}
Since we are dealing with a matrix Lie algebra, we can use the standard formula for the Maurer--Cartan form, which provides a $\fc$--coframe and the component $\alpha(\sum_i He^i)$. We can then use the $\fc$--coframe to write the $g$.
\end{proof}

In the above exponential coordinates, let us consider a curve $$\gamma(t)=\left[ \begin{smallmatrix}
0&\gamma_3& \gamma_4\\
\gamma_1& 0 & 0 \\
\gamma_2& 0 & 0
\end{smallmatrix} \right]\in \fc.$$
 We conclude from Lemma \ref{pfs-lem} that
\begin{gather*}
X^\gamma=\left[ \begin{smallmatrix}
\frac{1}{u}\\
0\\
0\\
0\\
0\\
0
\end{smallmatrix} \right],\ \ U^\gamma=
\left[ \begin{smallmatrix}
-\frac{u'}{u^2}\\
u_1\\
u_2\\
u_3\\
u_4\\
0
\end{smallmatrix} \right],\ \ 
A^\gamma=\left[ \begin{smallmatrix}
\frac{u^5-2u(u'')+4(u')^2}{2u^3}\\
a_1\\
a_2\\
a_3\\
a_4\\
-u^2
\end{smallmatrix} \right]
\end{gather*}

Let us compute the conformal Killing--Yano $2$--forms and determine the corresponding conserved quantities.
\begin{prop}\label{pfs-q}
In the case $n=2$, the following quantities are constant along the conformal circles of $(M,[g])$, where we write $a_iu_j-a_ju_i=r_{ij}$
\begin{gather*}
{ C_1}={\frac {1}{u  }}\big(- \gamma_3 ^{2
}{ u_1}u   ^{2}+
 ( -{ \gamma_4}  { u_2}  
  u^{2}- r_{31}+ r_{42} ) {
 \gamma_3}  +{ u_3}    u
   ^{2}+2{ \gamma_4}   
r_{23}\big),
\\
{ C_2}={\frac {1}{u  }}\big(-\gamma_4^{2}{ u_2}u^{2}+ ( -{ \gamma_3} { u_1}u^{2}+r_{31}- r_{42} ) { \gamma_4} +{ u_4}  u^{2}+2{ \gamma_3}   r_{14}\big),
\\
{ C_3}={\frac { 1}{u  }}\Big(\big(  (   \gamma_3^{2}{ u_1}  +{ \gamma_3}
    { \gamma_4}{ u_2} -{ u_3}   )   u^{2}-2{ \gamma_4}   r_{23}+{
 \gamma_3}   (  r_{31}- r_{42} ) 
 \big)\gamma_1 ^{2}+ 
\\
\big( 
 (  ( { \gamma_3}  { \gamma_4}{ u_1}  +  \gamma_4^{2}{ u_2}  -{ u_4} ) { \gamma_2}  +2{ \gamma_3}
  { u_1}  +{ \gamma_4} { u_2}   ) u^{2}+ (  (  r_{42}- r_{31} ) {
 \gamma_4}  -2{ \gamma_3}   
r_{14} ) { \gamma_2}  +
\\
 r_{31}-r_{42}
 \big) { \gamma_1}  +{ u_1}  
 ( { \gamma_2}  { \gamma_4}  +1
 )  u ^{2}-2{ \gamma_2}
   r_{14}\Big),
\\
{ C_4}={\frac {
 1}{u }}\Big(\big (  (\gamma_4^{2}{ u_2}+ { \gamma_3}  { \gamma_4} { u_1}     -{ u_4} ) u^{2}-2{ 
\gamma_3}   r_{14}- { \gamma_4}(  r_{31}- r_{42}
 )    \big) \gamma_2 ^{2}+ 
\\
\big(  (  (  
\gamma_3^{2}{ u_1}  +{ 
\gamma_3}    { \gamma_4} { u_2}
  -{ u_3}   ) { \gamma_1}
  +{ \gamma_3}  { u_1} +2{ \gamma_4}  { u_2}  
 )   u   ^{2}+ (   
 (  r_{31}- r_{42} ){ \gamma_3} -2{ 
\gamma_4}   r_{23}  ) { \gamma_1} -
\\
 r_{31}+ r_{42} \big) { \gamma_2}  +{
 u_2}   ( { \gamma_1}  { 
\gamma_3}  +1 )   u  
 ^{2}-2{ \gamma_1}   r_{23}\Big),
\\
{ C_5}={\frac { 1}{u }}\Big( \big( -2{ \gamma_1}   \gamma_3^{2}{ u_1}
  + (  ( -2{ \gamma_1}  {
 u_2}  -{ \gamma_2}  { u_1}
   ) { \gamma_4}  -2{ u_1}
   ) { \gamma_3}  -{ \gamma_2}
    \gamma_4^{2}{
 u_2}  +
\\
{ u_4}  { \gamma_2}
  +2{ u_3}  { \gamma_1} -{ \gamma_4}  { u_2}  
 \big)  u ^{2}+ (  ( -2
 r_{31}+2 r_{42} ) { \gamma_1}  +2{ 
\gamma_2}   r_{14} ) { \gamma_3}  +
\\
 ( 4{ \gamma_1}   r_{23}+{ 
\gamma_2}   (  r_{31}- r_{42} )  ) 
{ \gamma_4}  - r_{31}+ r_{42}\Big),
\\
{ C_6}={\frac { 1}{u }}\Big(\big( -  \gamma_3^{2}{ u_1}   u^{2}+ ( -{ \gamma_4}  { u_2}
   u^{2}-r_{31}+
 r_{42} ) { \gamma_3}  +{ u_3} u^{2}+2{ \gamma_4}
  r_{23} \big) { \gamma_2}  -
\\
{
 \gamma_3}  { u_2}  u^{2}+2 r_{23}\Big) ={ C_1}{ \gamma_2} -{\frac {1}{u }}\big({ \gamma_3
}  { u_2}   u^{2}-2 r_{23}\big),
\\
{
 C_7}={\frac { 1}{u }}\Big( \big( -  \gamma_4^{2}{ u_2}    u^{2}+ ( -{ \gamma_3}  { u_1}  u^{2}+ r_{31}- r_{42}
 ) { \gamma_4}  +{ u_4}  
  u^{2}+2{ \gamma_3} r_{14} \big) { \gamma_1}  -
\\
{ \gamma_4}
  { u_1}    u^{2}+2 r_{14}\Big) = { C_2}{ \gamma_1}  -{\frac {1}{u  }}\big({ \gamma_4}
  { u_1}    u^{2}-2 r_{14}\big),
\\
{ C_8}={
\frac { 1}{u  }}\Big( \big( -{ \gamma_1}    \gamma_3^{2}{ u_1}  + ( 
 ( -{ \gamma_1}  { u_2}  -2{
 \gamma_2}  { u_1}   ) { 
\gamma_4}  -{ u_1}   ) { 
\gamma_3}  -2{ \gamma_2}
 \gamma_4 ^{2}{ u_2}  +
\\
2
{ u_4}  { \gamma_2}  +{ u_3}
  { \gamma_1}  -
2{ \gamma_4}
  { u_2}   \big)  u^{2}+ (  ( - r_{31}+ r_{42} ) {
 \gamma_1}  +4{ \gamma_2}   
r_{14} ) { \gamma_3}  +
\\
 ( 2{ \gamma_1}
   r_{23}+2{ \gamma_2}   ( 
 r_{31}- r_{42} )  ) { \gamma_4}  +
 r_{31}- r_{42}\Big),
\\
{ C_9}=-{\frac {1}{u }}( r_{31}+
 r_{42})
\end{gather*}
\end{prop}
\begin{proof}
As before we use Maple to compute the data from Theorem \ref{thm-loc} for  conformal Killing--Yano $2$--forms. We obtain that $\mathbb{S}\subset  \wedge^3\mathbb{T}$ has the coordinates 
\begin{gather*}
(0, {\scriptstyle\frac32}(v_5+v_8), 0, v_1, 0, {\scriptstyle\frac32}(v_5+v_8), v_2, 0, v_3, v_4, v_2, -v_1,
\\
 0, v_4, v_5-v_8+v_9, 2v_6, -v_3, 2v_7, -v_5+v_8+v_9, 0)
\end{gather*}
 in the standard basis  $t_{[ijk]}, 1\leq i<j<k\leq 6$ of $\wedge^3\mathbb{T}$, i.e., $v_9$ corresponds to $t_{[246]}+t_{[356]}$. The representation $\Phi$ is the adjoint representation of $\fk$ on $$\left[ \begin {smallmatrix} v_9-v_5-v_8&v_3&v_4\\ v_1&v_9+v_5&v_6\\ v_2&v_7&v_9+v_8\end {smallmatrix} \right],$$  If we replace the exponential coordinates from the formula \eqref{parsec} with  our exponential coordinates, we get the following sections of the tractor bundle
\begin{gather*}
v_1\big(0, {\scriptstyle \frac32}x_3, 0, 1, 0, {\scriptstyle \frac32}x_3, 0, 0, -x_3^2, -x_3x_4, 0, -1, 0, -x_3x_4, x_3, 2x_4, x_3^2, 0, -x_3, 0\big)+
\\
v_2\big(0, {\scriptstyle \frac32}x_4, 0, 0, 0, {\scriptstyle \frac32}x_4, 1, 0, -x_3x_4, -x_4^2, 1, 0, 0, -x_4^2, -x_4, 0, x_3x_4, 2x_3, x_4, 0\big)+
\\
v_3\big(0,-{\scriptstyle\frac32}{ x_1} ( { x_1}{ x_3}+{ x_2x_4}+1
 ) ,0,-{ x_1}^{2},0,-{\scriptstyle \frac32}{ x_1} ( { x_1}
{ x_3}+{ x_2x_4}+1 ) ,-{ x_1x_2},0,
\\
 x_1^{2
} x_3^{2}+{ x_3} ( { x_2x_4}+2 ) { x_1
}+{ x_2x_4}+1,{ x_1}{ x_4} ( { x_1}{ 
x_3}+{ x_2}{ x_4}+1 ) ,-{ x_1x_2},x_1^{
2},0,\\
{ x_1}{ x_4} ( { x_1}{ x_3}+{ x_2}
{ x_4}+1 ) ,-{ x_1} ( { x_1}{ x_3}-{
 x_2x_4}+1 ) ,-2x_1^{2}{x_4},
-{ x_1}
 ( { x_1}x_3^{2}+
\\
2{ x_3}+{ x_2x_3x_4}
 ) -{x_2x_4}-1,-2{ x_2} ( { x_1}{ 
x_3}+1 ) ,{ x_1} ( { x_1}{ x_3}-{ 
x_2x_4}+1 ) ,0\big)+
\\
+v_4\big(0,-{\scriptstyle \frac32}{ x_2} ( { x_2}{ x_4}+{ x_1x_3}+1
 ) ,0,
-{ x_1x_2},0,-{\scriptstyle \frac32}{ x_2} ( { x_2}{
 x_4}+{ x_1x_3}+1 ) ,-x_2^{2},0,
\\
{ x_2}{
 x_3} ( { x_1}{ x_3}+{ x_2}{ x_4}+1
 ) ,{ x_2}{ x_4} ( { x_1}{ x_3}+{ 
x_2}{ x_4}+2 ) +{ x_1x_3}+1,- x_2^{2},{ 
x_1x_2},0,
\\
{ x_2}{ x_4} ( { x_1}{ x_3}+{
 x_2}{ x_4}+2 ) +{ x_1x_3}+1,{ x_2} ( 
{ x_2}{ x_4}-{ x_1x_3}+1 ) ,-2{ x_1}
 ( { x_2}{ x_4}+1 ),
\\
-{ x_2}{ x_3}
 ( { x_1}{ x_3}+{ x_2}{ x_4}+1 ) ,-2{
{ x_2}}^{2}{ x_3},{ x_2} ( -{ x_2}{ x_4}+
{ x_1x_3}-1 ) ,0 \big)
\\
v_5\big(0,3{ x_1}{ x_3}+{\scriptstyle \frac32}({ x_2x_4}+1),0,2{ x_1},0
,3{ x_1x_3}+{\scriptstyle \frac32}({ x_2x_4}+1),{ x_2},0,-{ x_3}
 ( 2{ x_1}{ x_3}+
\\
{ x_2}{ x_4}+2 ) ,-{
 x_4} ( 2{ x_1}{ x_3}+{ x_2}{ x_4}+1
 ) ,{ x_2},-2{ x_1},0,{ x_4} ( 2{ x_1
}{ x_3}+{ x_2}{ x_4}+1 ) ,
\\
2{ x_1x_3}-{
 x_2x_4}+1,4{ x_1x_4},{ x_3} ( 2{ x_1}{
 x_3}+{ x_2}{ x_4}+2 ) ,2{ x_2x_3},-2{
 x_1x_3}+{ x_2x_4}-1,0\big)
+
\\
v_6\big(0, {\scriptstyle\frac32}x_2x_3, 0, x_2, 0, {\scriptstyle\frac32}x_2x_3, 0, 0, -x_2x_3^2, -x_2x_3x_4-x_3, 0, -x_2, 0, -x_2x_3x_4-x_3,
\\
 x_2x_3, 2(x_2x_4+1), x_2x_3^2, 0, -x_2x_3, 0\big)+
\\
v_7\big(0, {\scriptstyle\frac32}x_1x_4, 0, 0, 0, {\scriptstyle\frac32}x_1x_4, x_1, 0,
 -x_4(x_1x_3+1), -x_1x_4^2, x_1, 0, 0,
\\
 -x_1x_4^2, -x_1x_4, 0, x_4(x_1x_3+1), 2(x_1x_3+1), x_1x_4, 0\big)+
\\
v_8\big(0, 3x_2x_4+{\scriptstyle\frac32}(x_1x_3+1), 0, x_1, 0, 3x_2x_4+{\scriptstyle\frac32}(x_1x_3+1), 2x_2, 0,
-x_3(x_1x_3+
\\
2x_2x_4+1),
-x_4(x_1x_3+2x_2x_4+2), 2x_2, -x_1, 0, -x_4(x_1x_3+2x_2x_4+2),
\\
x_1x_3-2x_2x_4-1,
 2x_1x_4, 
x_3(x_1x_3+2x_2x_4+1),
4x_2x_3, -x_1x_3+2x_2x_4+1, 0\big)\\
+v_9\big(0, 0, 0, 0, 0, 0, 0, 0, 0, 0, 0, 0, 0, 0, 1, 0, 0, 0, 1, 0\big),
\end{gather*}
After contracting these sections with $\Sigma^\gamma$ using the tractor metric induced by $\bg$, we obtain the claimed conserved quantities.
\end{proof}

Now, we can compute a co--homogeneity one $K$--orbit of conformal circles. Note that the conserved quantities implicitly describe all the conformal circles, but we need an ansatz on the conserved quantities to get explicit equations.

\begin{thm}\label{pbs-thm}
The following curves are conformal circles (up to parametrization) for all $c\in \R$, $c> 0$
\begin{gather*}
\gamma_1:=\frac{2 c^2-t}{2c^3}\\
\gamma_2:=-\frac{1}{2c^3t}\\
\gamma_3:=\frac1t\\
\gamma_4:=t
\end{gather*}
and their $K$--orbit is  of co--homogeneity one.
\end{thm}
\begin{proof}
We investigate the action of $Sl(3,\R)$ on the conserved quantities from Proposition \ref{pfs-q}. This leads us to us an ansatz
$$\big(0, 0, -{ \frac{\sqrt{2}}{2}}c^{\frac12}, 0, { \frac{\sqrt{2}}{2}}c^{\frac32}, 0, 0, -{ \frac{\sqrt{2}}{2}}c^{\frac32}, 0\big)$$ for the value of the conserved quantities, because under this ansatz, we can derive the following equalities for conformal circles from the conserved quantities
\begin{gather*}
\gamma_1'' = -{ \frac{\gamma_4''}{2\gamma_4\gamma_3c^3}},\ \  
\gamma_2'' = { \frac{\gamma_4''\gamma_4-2\gamma_4'^2}{2\gamma_4^3c^3}},\ \  
\gamma_3'' = -\frac{\gamma_3(\gamma_4''\gamma_4-2\gamma_4'^2)}{\gamma_4^2},
\\
 \gamma_1' = -{ \frac{\gamma_4'}{2\gamma_3\gamma_4c^3}},\ \  
\gamma_2' = { \frac{\gamma_4'}{2\gamma_4^2c^3}}, \ \ 
\gamma_3' = -\frac{\gamma_3\gamma_4'}{\gamma_4}
\\
 \gamma_1 = { \frac{2\gamma_3c^2-1}{2\gamma_3c^3}}, \ \ 
\gamma_2 = -\frac{1}{2\gamma_4c^3}.
\end{gather*}
Thus we can pick $\gamma_4$ arbitrarily, compute the remaining functions and act by $Sl(3,\R)$ on the curves to check that the orbit is  codimension one $K$--orbit. One can choose $\gamma_4$ to get parametrization by arc--length, but for simplicity, we have chosen  $\gamma_4=t$  in our claim.
\end{proof}

Let us prove that the Theorem \ref{pbs-thm} can be generalized to an arbitrary dimension.

\begin{thm}
The following curves are conformal circles (up to parametrization) for all $c\in \R,$ $c> 0$
\begin{gather*}
\gamma_1:=\frac{2 c^2-t}{2c^3}\\
\gamma_2:=-\frac{1}{2c^3t}\\
\gamma_i:=0, n\geq i>2\\
\gamma_{n+1}:=\frac1t\\
\gamma_{n+2}:=t\\
\gamma_i:=0, 2n\geq i>n+2
\end{gather*}
in the exponential coordinates
 $$ {\sf c}: \left[\begin{smallmatrix}
0& X_2^t\\
X_1& 0  \\
\end{smallmatrix} \right]\mapsto \left[\begin{smallmatrix}
1& 0\\
X_1& \id  \\
\end{smallmatrix} \right]\left[\begin{smallmatrix}
1& X_2^t\\
0& \id  \\
\end{smallmatrix} \right]o$$
and their $K$--orbit is of codimension one.
\end{thm}
\begin{proof}
We conclude for the claimed curve  that $X^\gamma$ is just the above $X^\gamma$ for $n=2$ competed by zeros to the correct size, and therefore, the same holds for $U^\gamma$. For $A^\gamma$, we need to look at the component of $\alpha(\sum_i He^i)$ and verify that it does not induce any non--zero difference from the above $A^\gamma$ for $n=2$. It is not hard computation to show that apart from the diagonal, we have $x_{n+j}dx_{i}$ which indeed evaluates to zero on our curve in the difference from the above $A^\gamma$. To check that we have a conformal circle we need to use the tractor connection one more time and check that the result is still linearly dependent on $X^\gamma,U^\gamma,A^\gamma$. In this case also the difference in the $\alpha_1$--parts of $\alpha$ appear, but this shifts the result by multiple of $U^\gamma$ and thus does not influence the linear independence, i.e., the claimed curves are conformal circles (up to parametrization).

By the structure of $Sl(n+1,\R)$, it suffices to check the property that the $K$--orbit is of codimension one for $n=3$, because then it extends for the general $n$ by similar computations. It is not hard to achieve this using Maple.
\end{proof}

\appendix

\section{First BGG operators on homogeneous conformal geometries in exponential coordinates} \label{bgg-const}

Not all of the first BGG operators are described explicitly in the literature. Let us provide some details on how to describe them on homogeneous conformal geometries. Let us start by summarizing the necessary ingredients from the proof of Theorem \ref{thm-loc} in the local coordinates ${\sf c}:\fc \to M.$
\noindent \\[2mm] {{\bf (1)}}
 We have the formula $$\nabla^{\rho\circ \alpha}s= ds+\rho\circ \alpha\circ \big((\alpha_{-1})^{-1}\circ (e^1,\dots,e^n)+\sum_i H_ie^i)\big)(s)$$ for the action of the tractor connection on a section of the tractor bundle $\mathcal{V}=\fc\times \mathbb{V}$ represented by a function $s: \fc \to \mathbb{V}$.
\noindent \\[2mm] {{\bf (2)}}
 We need the coefficients of the polynomial $Q$ in $\partial^*\nabla^{\rho\circ \alpha}$, that is either provided by the representation theory, \cite{casimir}, or by solving 
$$\partial^* \nabla^{\rho\circ \alpha}(\id -Q \partial^* \nabla^{\rho\circ \alpha})s=0$$
 for a general polynomial  $Q$ in $\partial^*\nabla^{\rho\circ \alpha}$ and all $s: \fc\to \mathbb{X}.$ Let us note that if $\lambda$ is irreducible representation of $G$, then the degree 
of $Q$ is the number of irreducible $\frak{co}(p,q)$--submodules in $\mathcal{V}$ minus $2$.
\noindent \\[2mm] {{\bf (3)}}
 We need the projection $\pi_1: Ker(\partial^*)\subset \Omega^1(\mathbb{V})\to \mathcal{H}^1(\mathbb{V})$ that is obtained from the Kostant's version of Bott--Borel--Weil Theorem, \cite[Section 3.3]{parabook}.
\\[2mm] 
Having these we can write the first BGG operator $$\mathcal{D} s=\pi_1 \nabla^{\rho\circ \alpha}(\id -Q \partial^* \nabla^{\rho\circ \alpha})s$$ for $s: \fc\to \mathbb{X}.$ If $s_i$ are components of $s$ in a basis of $\mathbb{X}$ corresponding to sections $\sigma^i$ induced by a $\fc$--(co)frame, then $ds$ gets replaced by $d(s_i\sigma^i)-s_id\sigma^i.$ For example, if the $\fc$--frame has the form  $e_i=\sigma^i_j\partial_{a_j}$ in  $(a_1,\dots,a_n)$--coordinates on $\fc$, then for $\mathbb{X}=\R^n$ the section $s$ corresponds to the vector field $S=s_i\sigma^i_j\partial_{a_j}$ and $ds$ gets replaced by $(d(s_i\sigma^i_j)-s_id\sigma^i_j)\partial_{a_j}$. Similarly, if $\epsilon$ induced by $\fc$--coframe has the form $det(\sigma)^{-1}da_1\wedge \dots \wedge da_n$ in  $(a_1,\dots,a_n)$--coordinates, then for $\mathbb{X}=\mathbb{R}[w]$ the section $s$ corresponds to the conformal density $\tau=s\cdot \epsilon^{\frac{-w}{n}}$ and $ds$ gets replaced by $d\tau-
\frac{w}{n}Tr(\sigma^{-1}d\sigma)\tau.$
In particular, 
\begin{gather*}
\nabla S=(d(s_i\sigma^i_j)-s_id\sigma^i_j)\partial_{a_j}+[(\alpha_0\circ (\alpha_{-1})^{-1}\circ(da_1,\dots,da_n)(\sigma^t)^{-1}))(s)\\
+(d\iota_0\circ (da_1,\dots,da_n)(\sigma^t)^{-1}(H_1,\dots,H_n)^t)(s)]_i\sigma^i_j\partial_{a_j}\\
\nabla \tau=d\tau-\frac{w}{n}Tr(\sigma^{-1}d\sigma)\tau-wa_ke^k \tau
\end{gather*}
are the actions of the corresponding Weyl connections, where $a_k$ is from the formula for the Cartan connection according to decomposition \eqref{Cartan_connection}. Note that if $e_i$ is an orthonormal frame for a metric preserved by the Weyl connection, then $Tr(\sigma^{-1}d\sigma)=-na_ke^k$ and thus $\nabla \tau=d\tau$.

\begin{exam}\label{exam1}
We investigate for $n=4$ the tractor bundle that is the trace--free symmetric product of the standard tractor bundle, i.e., $\mathbb{V}=\bigodot^2_0\mathbb{T}$. The grading on $\mathbb{V}$ has the following $(1,4,1)$--block structure
$$\left[ \begin{smallmatrix}
\mathbb{V}_{2}& *& *\\ \mathbb{V}_1 & \mathbb{V}_0 & *\\ *& \mathbb{V}_{-1}& \mathbb{V}_{-2}\end{smallmatrix} \right],$$
 where the $*$--entries depend on the other entries and $\mathbb{X}=\mathbb{V}_{-2}=\R[2]$.
There are six irreducible components in $\mathbb{V}$ and thus $Q$ is a polynomial of degree four. We compute  $$Q=-{ \frac{35}{24}}-{\frac79}\partial^*\nabla^{\rho\circ \alpha}-{\frac{3}{16}}(\partial^*\nabla^{\rho\circ \alpha})^2-{ \frac{1}{48}}(\partial^*\nabla^{\rho\circ \alpha})^3-{ \frac{1}{1152}}(\partial^*\nabla^{\rho\circ \alpha})^4$$ (one can for simplicity assume that $\fk=\R^4$, $\alpha=\alpha_{-1}$ is identity on $\R^4$ and thus $H_i=0, i=1,2,3,4$, and do the computations explicitly for general polynomial of degree four). The projection $\pi_1$ is then a projection $\R^{4*}\otimes \mathbb{V}\to \bigodot^3_0 \R^{4*}[2]$ given by the trace--free part of complete symmetrization of $\R^{4*}\otimes \mathbb{V}_0=\R^{4*}\otimes  \bigodot^2 \R^{4*}[2]$. Finally, let us mention that the leading terms are third order partial derivatives that are trace--free and there are uniquely determined lower order terms making this operator conformally invariant. Thus finding the solutions of this BGG operator classically involves solving sixteen third--order equations for a function of four variables.
\end{exam}

\section{Example of construction of prolongation connection} \label{apendix}
The construction of the prolongation connection is given by the algorithm from Section \ref{section3.3}, however, the computation can be tedious. Therefore, we provide just one example of this construction, for the case of the G\"odel metric and the tractor bundle $\mathbb{V}=\bigodot^2_0\mathbb{T}$ and we follow the notation of Section \ref{section3.3}, i.e.,
$$\left[ \begin{smallmatrix}
\mathbb{V}_{2}& *& *\\ \mathbb{V}_1 & \mathbb{V}_0 & *\\ *& \mathbb{V}_{-1}& \mathbb{V}_{-2}\end{smallmatrix} \right],\ \ \ 
\left[ \begin {smallmatrix} w&*&*&*&*&*\\ u_{{1}}
&r_{{1}}&*&*&*&*\\ u_{{2}}&r_{{5}}&r_{{2}}&*&*&*
\\ u_{{3}}&r_{{8}}&r_{{6}}&r_{{3}}&*&*
\\ u_{{4}}&r_{{9}}&r_{{10}}&r_{{7}}&r_{{4}}&*
\\ -r_{{9}}-\frac12 (r_{{2}}+ r_{{3}})&v_{{1}}&v_{{2}}&v_{{3}}&v_{{4}}&s
\end {smallmatrix} \right] 
.
$$
We start with the tractor connection and according to the algorithm, we construct $c_i$ and $\psi_i$ to find $\Psi$ and thus the prolongation connection. It turns out that our computation involves  seven steps. 
\noindent \\[2mm] {\bf (1)}
In the homogeneity $2$, we find $c_1=6$ and $\psi_1(x_1,x_2,x_3,x_4,0)$ then equals to a  $\frac16$--multiple of
\[
-\frac23 \cdot\left[ \begin {smallmatrix} 0&*&*&*&*&*\\ 
 (( r_{{2}}+r_{{3}}-2r_{{9}} ) x_{{1}}-x_{{2}}r_{{5}}
-x_{{3}}r_{{8}})+x_{{4}}r_{{1}}&2x_{{1}}v_{{1}}&
*&*&*&*
\\ 
( r_{{9}}-r_{{3}} ) x_{{2
}}-\frac12(x_{{1}}r_{{10}}+x_{{4}}r_{{5}})+x_{{3}}r_{{6}}&-\frac12z_{12}&z_{14}-2x_{{3}}v_{{3}}&
*&*&*
\\ 
 ( r_{{9}}-r_{{2}}
 ) x_{{3}}-\frac12(x_{{1}}r_{{7}}-x_{{4}}r_{8})+x_{{2}}r_{{6}}
&
-\frac12z_{13}&
z_{23}&
z_{14}-2x_{{2}}v_{{2}}&
*&*
\\ 
\frac12 (( r_{{2}}+r_{{3}}-2r_{{9}} ) x
_{{4}}-x_{{2}}r_{{10}}-x_{{3}}r_{{7}})+x_{{1}}r_{{4}}&
x_{{2}}v_{{2}}+x_{{3}}v_{{3}}-z_{14}&
-\frac12z_{24}&
-\frac12z_{34}&
2x
_{{4}}v_{{4}}&
*
\\ 
0&0&0&0&0&0\end {smallmatrix} \right], 
\]
where we write $z_{ij}=v_ix_j+v_jx_i$.
\noindent \\[2mm] {\bf (2)}
In the homogeneity $3$, we find $c_2=6$ and $(\psi_2-\psi_1)(x_1,x_2,x_3,x_4,0)$ then equals to a  $\frac16$--multiple of
\[
 {\sqrt{2}}\cdot\left[ \begin {smallmatrix}  ( 2r_{{7}}-4r_{{8}} ) 
x_{{2}}+4 ( r_{{5}}-\frac12r_{{10}} ) x_{{3}}&*&*&*&*&*
\\ \frac34(x_{{3}}v_{{2}}-x_{{2}}v_{{3}})&0&*&*&*&*
\\ \frac14 ( x_{{4}}-2x_{{1}} ) v_{{3}}-
2 ( v_{{1}}-\frac12v_{{4}} ) x_{{3}}&-\frac34x_{{3}}s&0&*&*&*
\\ \frac14 ( 2x_{{1}}-x_{{4}} ) v_{{2}}+2
 ( v_{{1}}-\frac12v_{{4}} ) x_{{2}}&\frac34x_{{2}}s&0&0&*&*
\\ \frac32(x_{{2}}v_{{3}}-x_{{3}}v_{{2}})&0&\frac32x
_{{3}}s&-\frac32x_{{2}}s&0&*\\ \noalign{\medskip}0&0&0&0&0&0\end {smallmatrix}
 \right].
\]
\noindent \\[2mm] {\bf (3)}
Next, in the homogeneity $3$, we find $c_3=12$ and $(\psi_3-\psi_2)(x_1,x_2,x_3,x_4,0)$ then equals to a  $\frac1{12}$--multiple of
\[
\frac{\sqrt{2}}{6}\cdot \left[ \begin {smallmatrix}  ( 10r_{{8}}-5r_{{7}} ) 
x_{{2}}-10 ( r_{{5}}-\frac12r_{{10}} ) x_{{3}}&*&*&*&*&*
\\ \frac12(x_{{2}}v_{{3}}-x_{{3}}v_{{2}})&0&*&*&*&*
\\ \frac12 ( 10x_{{1}}-5x_{{4}} ) v_{{3
}}+6 ( v_{{1}}-\frac12v_{{4}} ) x_{{3}}&0&0&*&*&*
\\ \frac12 ( -10x_{{1}}+5x_{{4}} ) v_{{
2}}-6 ( v_{{1}}-\frac12v_{{4}} ) x_{{2}}&0&0&0&*&*
\\ x_{{3}}v_{{2}}-x_{{2}}v_{{3}}&0&0&0&0&*
\\ 0&0&0&0&0&0\end {smallmatrix} \right].
\]
\noindent \\[2mm] {\bf (4)}
Next, in the homogeneity $3$, we find $c_4=8$ and $(\psi_4-\psi_3)(x_1,x_2,x_3,x_4,0)$ then equals to a  $\frac18$--multiple of
\[
\frac{\sqrt{2}}{9}\cdot \left[ \begin {smallmatrix} \frac14 ( 14r_{{8}}-7r_{{7}}
 ) x_{{2}}-\frac72 ( r_{{5}}-\frac12r_{{10}} ) x_{{3}}&*&
*&*&*&*\\ \frac12(x_{{3}}v_{{2}}-x_{{2}}v_{{3}})&0
&*&*&*&*\\ \frac12v_{{3}} ( 2x_{{1}}-x_{{4}}
 ) &0&0&*&*&*\\ \frac12 ( -2x_{{1}}+x_{{4
}} ) v_{{2}}&0&0&0&*&*\\ x_{{2
}}v_{{3}}-x_{{3}}v_{{2}}&0&0&0&0&*\\ 0&0&0&0&0&0\end {smallmatrix}
 \right] 
.
\]
\noindent \\[2mm] {\bf (5)}
Next, in the homogeneity $3$, we find $c_5=12$ and $(\psi_5-\psi_4)(x_1,x_2,x_3,x_4,0)$ then equals to a  $\frac1{12}$--multiple of
\[
\frac{\sqrt{2}}{4}\cdot
\left[ \begin {smallmatrix} \frac12 ( -2r_{{8}}+r_{{7}}
 ) x_{{2}}+ ( r_{{5}}-\frac12r_{{10}} ) x_{{3}}&*&*&*&*
&*\\ 0&0&*&*&*&*\\ 0&0&0&*&*&*
\\ 0&0&0&0&*&*\\ 0&0&0&0&0&*
\\ 0&0&0&0&0&0\end {smallmatrix} \right] 
\]
\noindent \\[2mm] {\bf (6)}
In the homogeneity $4$, we find $c_6=8$ and $(\psi_6-\psi_5)(x_1,x_2,x_3,x_4,0)$ then equals to a  $\frac1{8}$--multiple of
\[
 \frac19\cdot\left[ \begin {smallmatrix}  ( 16x_{{1}}-8x_{{4}} ) 
v_{{1}}+ ( -8x_{{1}}+4x_{{4}} ) v_{{4}}+12(v_{{2}}x_{{
2}}+v_{{3}}x_{{3}})&*&*&*&*&*\\ s ( 5x_{{1
}}-2x_{{4}} ) &0&*&*&*&*\\ -5x_{{2}}s&0&0&
*&*&*\\ -5x_{{3}}s&0&0&0&*&*\\ 
 ( -8x_{{1}}+5x_{{4}} ) s&0&0&0&0&*
\\ 0&0&0&0&0&0\end {smallmatrix} \right] 
\]
\noindent \\[2mm] {\bf (7)}
Finally, in the homogeneity $4$, we find $c_7=12$ and $(\psi_7-\psi_6)(x_1,x_2,x_3,x_4,0)$ then equals to a  $\frac1{12}$--multiple of
\[
 \frac19\cdot\left[ \begin {smallmatrix} -v_{{1}}x_{{4}}-v_{{2}}x_{{2}}-v_{{3}}x
_{{3}}-v_{{4}}x_{{1}}&*&*&*&*&*\\ 0&0&*&*&*&*
\\ 0&0&0&*&*&*\\ 0&0&0&0&*&*
\\ 0&0&0&0&0&*\\ 0&0&0&0&0&0
\end {smallmatrix} \right] 
\]
Let us note that there is nothing in homogeneity $5$ and the computation is finished.

Altogether, the prolongation connection $\na^\Phi$ is given by $\Phi=\rho \circ \alpha +\psi_7$, where one gets $\psi_7$ by adding appropriate multiples of the above matrices.

\end{document}